\documentclass[10pt, a4paper]{article}
\usepackage{amsfonts}
\usepackage{amsmath}
\usepackage{todonotes}
\usepackage{graphicx,color}
\usepackage{subfig}
\usepackage[unicode=true, bookmarks=true, colorlinks=true, urlcolor= blue, citecolor=blue, pdfauthor={Nikolai Leonenko, Alois Pichler},
pdftitle={Fractional Poisson Process}]{hyperref}
\usepackage[inline]{enumitem}
\usepackage{dsfont}

\setlist[enumerate,1]{label=(\roman*)}  
\DeclareMathOperator{\AVaR}{\mathsf {AV@R}}
\DeclareMathOperator{\VaR}{\mathsf{V@R}}
\DeclareMathOperator{\sign}{sign}
\newtheorem{theorem}{Theorem}

\newtheorem{corollary}[theorem]{Corollary}

\newtheorem{definition}[theorem]{Definition}

\newtheorem{lemma}[theorem]{Lemma}

\newtheorem{proposition}[theorem]{Proposition}
\newtheorem{remark}[theorem]{Remark}

\newenvironment{proof}[1][Proof]{\textbf{#1.} }{\ \rule{0.5em}{0.5em}}
\input{tcilatex}

\DeclareMathOperator{\E}{\mathds E} 
\DeclareMathOperator{\EVaR}{\mathsf {EV@R}}

\begin{document}

\begin{center}
	\bigskip
	
	\bigskip \textbf{FRACTIONAL\ RISK\ PROCESS\ IN\ INSURANCE\\[2em]}
	\textbf{Arun KUMAR$^{1}$ Nikolai LEONENKO$^{2}$  Alois
		PICHLER$^{3}$}
	
\end{center}

\begin{abstract}
The Poisson process suitably models the time of successive events and thus has numerous applications in statistics, in economics, it is also fundamental in queueing theory. 
Economic applications include trading and nowadays particularly high frequency trading. Of outstanding importance are applications in insurance, where arrival times of successive claims are of vital importance.

It turns out, however, that real data do not always support the genuine Poisson process. This has lead to variants and augmentations such as time dependent and varying intensities, for example. 

This paper investigates the fractional Poisson process. We introduce the process and elaborate its main characteristics. The exemplary application considered here is the Carm{\'e}r--Lundberg theory and the Sparre Andersen model. 
The fractional regime leads to initial economic stress.  On the other hand we demonstrate that the average capital required to recover a company after ruin does not change when switching to the fractional Poisson regime. We finally address particular risk measures, which allow simple evaluations in an environment governed by the fractional Poisson process.



\end{abstract}
%
%
%
%

$^{1}$Department of Mathematics, Indian Institute of Technology Ropar, Rupnagar, Punjab 140001, India

\href{arun.kumar@iitrpr.ac.in}{arun.kumar@iitrpr.ac.in}
\medskip

$^{2}$Cardiff School of Mathematics, Cardiff University, Senghennydd Road,
Cardiff
CF24 4AG, UK

\href{LeonenkoN@Cardiff.ac.uk}{LeonenkoN@Cardiff.ac.uk}
\medskip

$^{3}$Chemnitz University of Technology, Faculty of Mathematics, Chemnitz,
Germany

\href{alois.pichler@math.tu-chemnitz.de}{%
alois.pichler@math.tu-chemnitz.de}

\bigskip\bigskip


Mathematics Subject Classification (1991):

\bigskip

\bigskip

\section{Introduction}
Companies, which are exposed to random orders or bookings, face the threat of ruin. But what happens if bookings fail to appear and ruin actually happens?
The probability of this event cannot be ignored and is of immanent and substantial importance from economic perspective. Recall that some nations have decided to bail-out banks, so a question arising naturally is how much capital is to be expected to bail-out a company after ruin, if liquidation is not an option. 

This paper introduces the average capital required to recover a company after ruin. It turns out that this quantity is naturally related to classical, convex risk measures.
The corresponding arrival times of random events as bookings, orders, or stock orders on an exchange are often comfortably modelled by a Poisson process --- this process is a fundamental building block to model the economy reality of random orders.

From a technical modelling perspective, the inter-arrival times of the Poisson process are exponentially distributed and independent \cite{Khinchin1969}. However, real data do not always support this property or feature (see, e.g., \cite{Grandell1991, KLS, Raberto2002, Scalas2004} and references therein). 
 To demonstrate the deviation we display arrival times of two empirical data sets consisting of oil futures high frequency trading (HFT) data and the classical Danish fire insurance data. 
Oil futures HFT data is considered for the period from July 9, 2007 until August 9, 2007. There is a total of 951,091 observations recorded over market opening hours. 
The inter-arrival times are recorded in seconds for both up-ticks and down-ticks (cf.\ Figures~\ref{fig:1e} and~\ref{fig:1c}). 
In addition, the Danish fire insurance data is available for a period of ten years 
until December 31, 1990, consisting of 2167 observations: 
the time unit for inter-arrival times in Figure~\ref{fig:1a} is days.
The charts in Figure~\ref{fig:1} display the survival function of inter-arrival times of the data on a logarithmic scale and compares them with the exponential distribution. The data are notably not aligned, as is the case for the Poisson process. By inspection it is thus evident that arrival times of the data apparently do not follow a genuine Poisson process.
\begin{figure}[ht!]
	\begin{center}
		\subfloat[Oil futures HFT up-ticks]{\includegraphics[width= 0.33\textwidth]{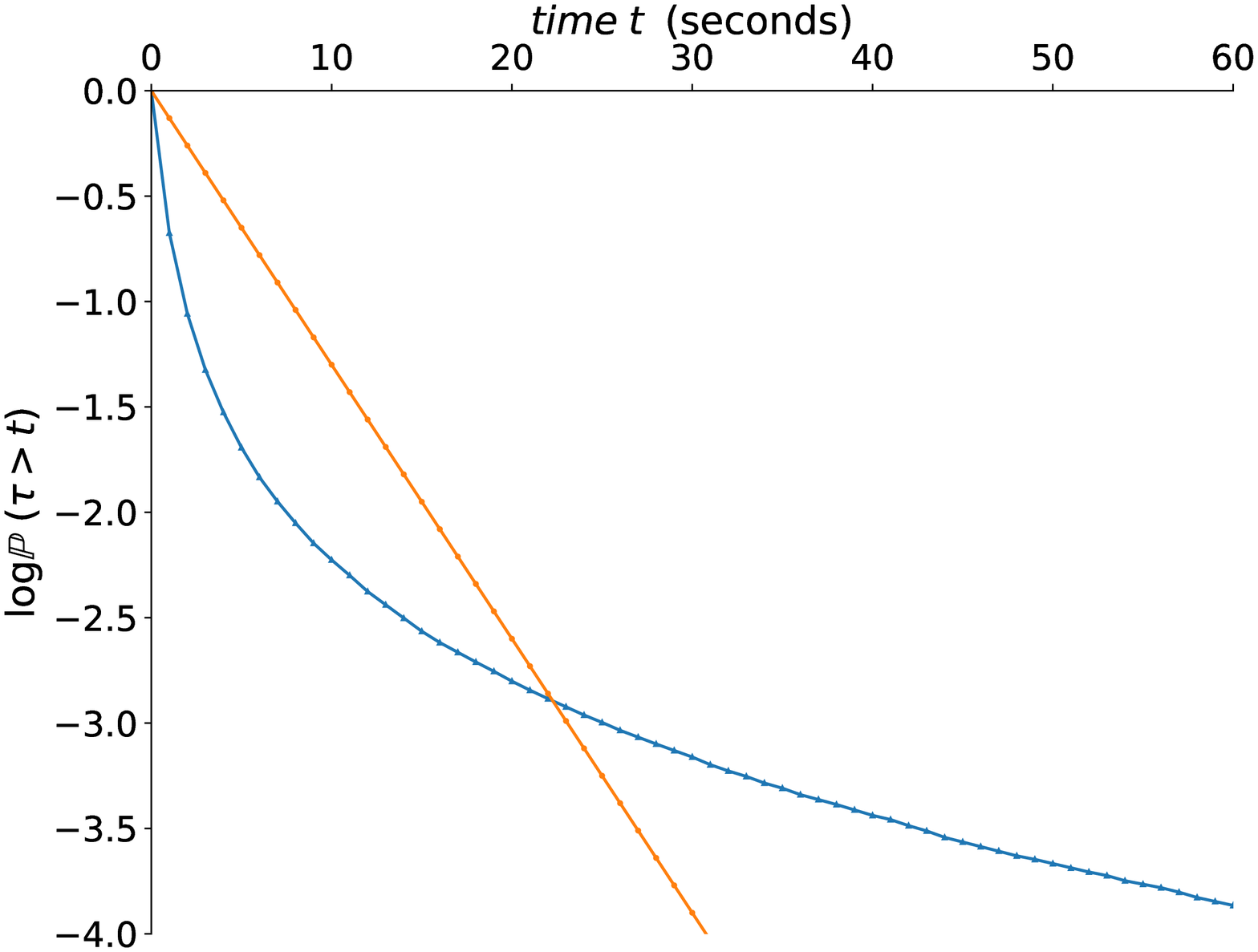}\label{fig:1e}}
		\subfloat[Oil futures HFT down-ticks]{\includegraphics[width= 0.33\textwidth]{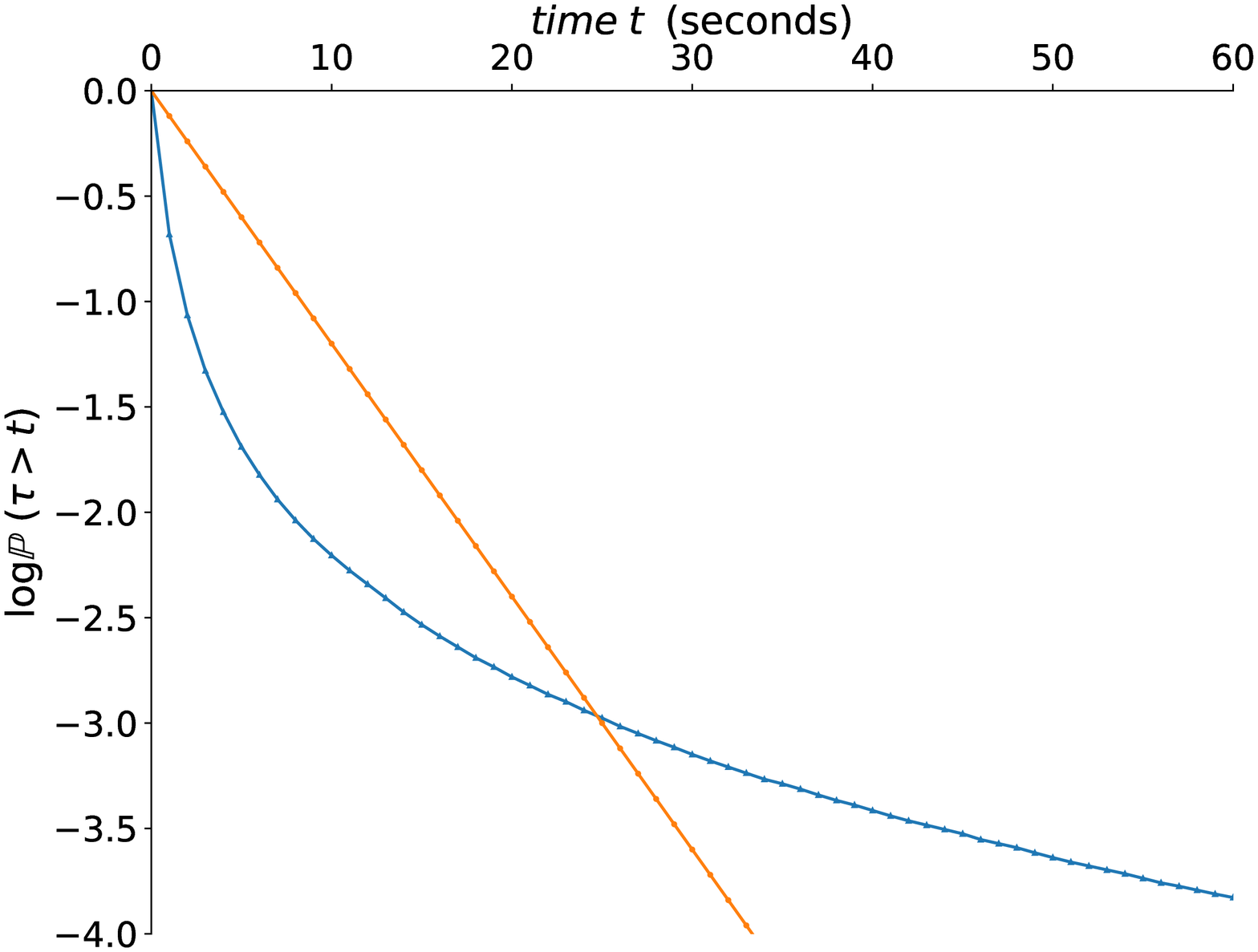}\label{fig:1c}}    
		\subfloat[Danish fire insurance data]{\includegraphics[width= 0.33\textwidth]{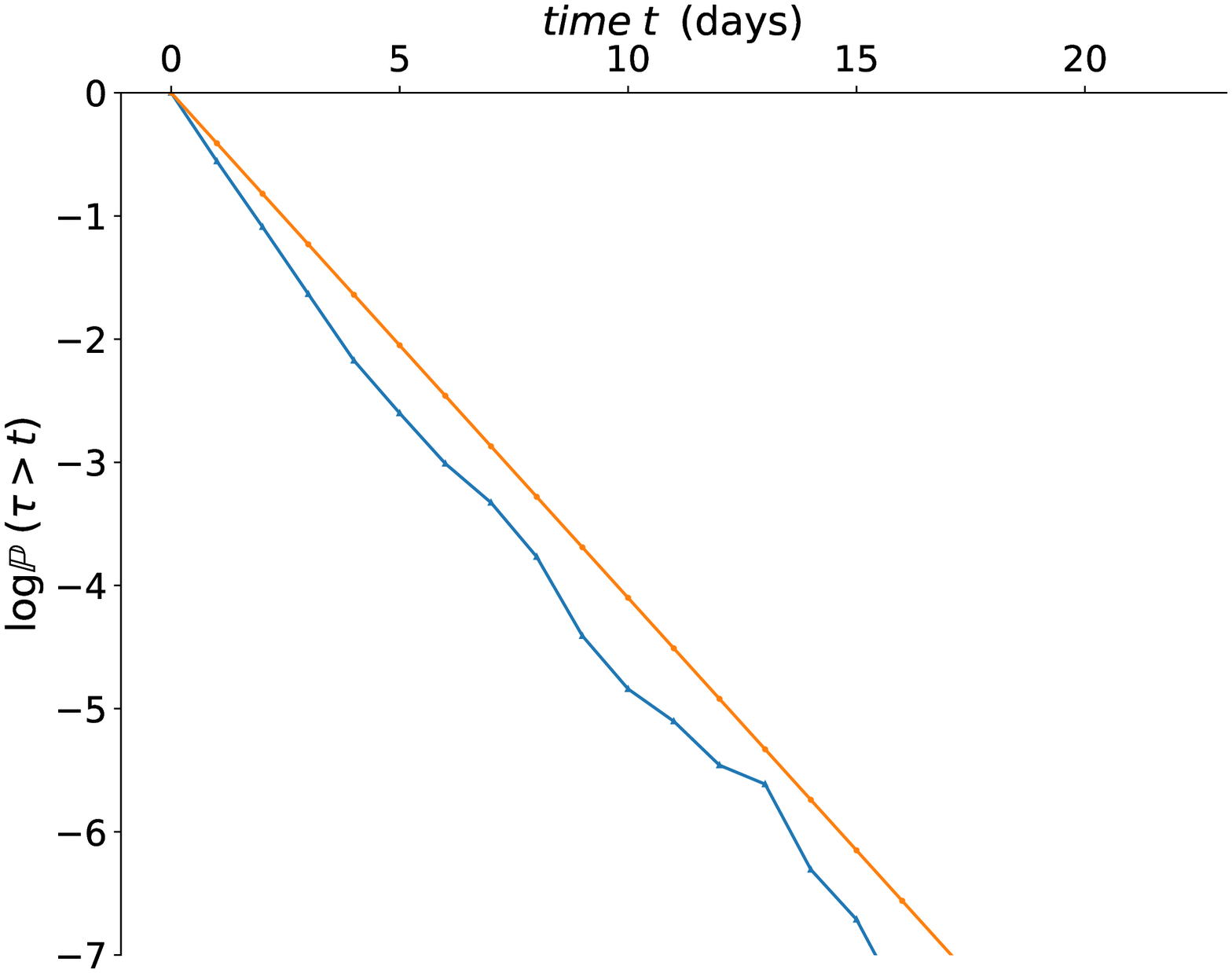}\label{fig:1a}}
		\caption{\label{fig:1}Survival functions for oil futures high frequency trading data and Danish fire insurance data, compared with the corresponding exponential distribution's survival functions}
	\end{center}
\end{figure}

This paper features 
the \emph{fractional} Poisson process. It is a parametrized stochastic counting process which includes the classical Poisson process as a special, perhaps extreme case, but with dependent arrival times. We will see that the fractional Poisson process correctly captures main characteristics of the data displayed in Figure~\ref{fig:1}.


Arrival times of claim events are of striking economic importance for insurance companies as well and our discussion puts a slight focus on this economic segment.
The Cram{\'e}r--Lundberg model, to give an example, genuinely builds on the Poisson process. We replace the Poisson process by its fractional extension and investigate the new economic characteristics. 
We observe that the fractional Poisson process exposes companies to higher initial stress. But conversely, we demonstrate that the ruin probability and the average capital required to recover a firm are identical for the classical and the fractional Poisson process. In special cases we can even provide explicit relations.

\paragraph{Outline of the paper.}
The following two Sections~\ref{sec:2} and~\ref{sec:3} are technical, they introduce and discuss the fractional Poisson process by considering L{\'e}vy subordinators.
Section~\ref{sec:4} introduces the risk process based on the fractional Poisson process, while the subsequent Section~\ref{sec:5} addresses insurance in a dependent framework. 
This section contains a discussion of average capital injection to be expected in case of ruin. We finally discuss the risk process for insurance companies and conclude in Section~\ref{sec:7}.

\section{Subordinators and inverse subordinators}
\label{sec:2}

This section considers the inverse subordinator first and introduces the Mittag-Leffler distribution. This distribution is essential in defining the fractional Poisson process. Relating the subordinators to $\alpha$-stable distributions enables to simulate trajectories of the fractional Poisson process. Some results here follow Bingham \cite{B}, Veillette and Taqqu \cite{VTM1}, \cite{VTM2}, \cite{MS1}.

\subsection{Inverse subordinator}

Consider a non-decreasing L\'{e}vy process $L=\{L(t),\ t\geq 0\},$ starting
from $0,$ which is continuous from the right with left limits (c{\`a}dl{\`a}%
g), continuous in probability, with independent and stationary increments.
Such process is known as L\'{e}vy subordinator with Laplace exponent 
\begin{equation}  \label{LS}
\phi (s)=\mu s+\int_{(0,\infty )}(1-e^{-sx})\Pi (dx),\ s\geq 0,
\end{equation}
where $\mu \geq 0$ is the drift and the L\'{e}vy measure $\Pi $ on $%
\mathbb{R}_{+}\cup \left\{ 0\right\} $ satisfies 
\begin{equation}
	\int\limits_{0}^{\infty }\min (1,x)\Pi (dx)<\infty .  \label{LS1}
\end{equation}
This means that 
\begin{equation}
\E e^{-s\,L(t)}=e^{-t\,\phi (s)},\ s\geq 0.  \label{LS2}
\end{equation}

The inverse subordinator $Y(t),\ t\geq 0,$ is the first-passage time of $L$, i.e.,
\begin{equation}  \label{LS3}
Y(t):=\inf \left\{ u\geq 0\colon L(u)>t\right\},\quad t\geq 0.
\end{equation}%
The process $Y(t),\ t\geq 0,$ is non-decreasing and its sample paths are almost surely (a.s.)\ continuous if $L$ is strictly increasing. Also $Y$ is, in general,
non-Markovian with non-stationary and non-independent increments.

We have 
\begin{equation}
\left\{ L(u_{i})<t_{i},\ i=1,\ldots ,n\right\} =\left\{Y(t_{i})>u_{i},\
i=1,\ldots ,n\right\}\label{IS1}
\end{equation}%
and for any $z>0$ 
\begin{equation*}
\mathrm{P}\left\{ Y(t)>x\right\} \leq \mathrm{P}\left\{ L(x)\leq t\right\} = 
\mathrm{P}\left\{ e^{-zL(x)}\geq e^{-tz}\right\} \leq \exp \{tz-x\phi (z)\},
\end{equation*}%
which implies that all moments are finite, i.e., $\E Y^{p}(t)<\infty$ for any $p>0$.

Let $U^{(p)}(t):=\E Y^{p}(t)$ and the renewal function 
\begin{equation}
U(t):=U^{(1)}(t)=\E Y(t),  \label{IS2}
\end{equation}
and let 
\begin{equation*}
	H_{u}(t):=\mathrm{P}\left\{ Y(u)<t\right\} .
\end{equation*}
Then for $p>-1$, the Laplace transform of $U^{(p)}(t)$ is given by 
\begin{equation*}
\tilde{U}^{(p)}(s)=\int_0^\infty U^{(p)}(t)e^{-st}dt=\Gamma (1+p)/[s\phi ^{p}(s)],
\end{equation*}
in particular, for $p=1$ 
\begin{equation}
\int_0^\infty e^{-st}dH_{u}(t)=e^{-t\phi
(s)},  \label{IS4}
\end{equation}
and 
\begin{equation}
\tilde{U}(s)=\int_0^\infty U(t)e^{-st}dt=%
\frac{1}{s\phi (s)};  \label{IS5}
\end{equation}
thus, $\tilde{U}$ characterizes the inverse process $Y$, since $\phi$
characterizes $L.$

Therefore, we get a covariance formula (cf.\ \cite{VTM1}, \cite{VTM2}) 
\begin{equation}  \label{IS6}
\mathrm{Cov}\big(Y(s),Y(t)\big)=\int_{0}^{\min(s,t)}\big(U(s-\tau )+U(t-\tau
)\big)dU(\tau )-U(s)U(t).
\end{equation}

The most important example is considered in the next section, but there are
some others.

\subsection{Inverse stable subordinator}

Let $L_{\alpha }=\{L_{\alpha }(t),t\geq 0\}$ be an ${\alpha}$\nobreakdash-stable
subordinator with Laplace exponent $\phi (s)=s^{{\alpha} },\ 0<{\alpha} <1$, whose density $%
g(t,x)$ is such that $L_{\alpha}(1)$ has pdf 
\begin{align}\label{eq:10}
g_{{\alpha} }(x)=g(1,x)&=\frac{1}{\pi }\mathop{\displaystyle \sum}
\limits_{k=1}^{\infty }(-1)^{k+1}\frac{\Gamma ({\alpha} k+1)}{k!}\frac{1}{
x^{{\alpha} k+1}}\sin (\pi k{\alpha} ).
\end{align}%
%
%
Then the inverse stable subordinator

\begin{equation*}
Y_{{\alpha} }(t)=\inf \{u\geq 0\colon L_{{\alpha} }(u)>t\}
\end{equation*}
has density 
\begin{equation}  \label{D1}
f_{{\alpha} }(t,x)=\frac{t}{{\alpha} }x^{-1-\frac{1}{{\alpha} }}g_{{\alpha}
}(tx^{- \frac{1}{{\alpha} }}),\ x\geq 0,\ t>0.
\end{equation}
Alternative form for the pdf $f_\alpha$ in term of $H$-function and infinite series are discussed in \cite{KKV, KV}.

Its paths are continuous and nondecreasing. For ${\alpha} =1/2,$ the inverse
stable subordinator is the running supremum process of Brownian motion, and
for ${\alpha} \in (0,1/2]$ this process is the local time at zero of a
strictly stable L\'{e}vy process of index ${\alpha} /(1-{\alpha} ).$

Let 
\begin{equation}
E_{{\alpha} }(z):=\sum_{k=0}^{\infty }\frac{z^{k}}{\Gamma ({\alpha} k+1)},\ {%
\alpha} >0,\ z\in \mathbb{C}  \label{ML1}
\end{equation}%
be the \emph{Mittag-Leffler function} (cf.\ \cite{HMS}) and recall the following:

\begin{enumerate}
\item The Laplace transform of the Mittag-Leffler function is of the form 
\begin{equation*}
\int_{0}^{\infty }e^{-st}E_{{\alpha} }(-t^{{\alpha} })dt=\frac{s^{{\alpha}
-1}}{ 1+s^{{\alpha} }},\ 0<{\alpha} <1,\ t\geq 0,\ s>0.
\end{equation*}

\item The Mittag-Leffler function is a solution of the fractional equation: 
\begin{equation*}
\mathrm{D}_{t}^{{\alpha} }E_{{\alpha} }(at^{{\alpha} })=aE_{{\alpha} }(at^{{%
\alpha} }),\ 0<{\alpha} <1,
\end{equation*}
where the fractional Caputo-Djrbashian derivative $\mathrm{D}_{t}^{{\alpha}
} $ is defined as (see~\cite{MS1}) 
\begin{equation}
\mathrm{D}_{t}^{{\alpha} }u(t)=\frac{1}{\Gamma (1-{\alpha} )}\int_{0}^{t}%
\frac{ du(\tau )}{d\tau }\frac{d\tau }{(t-\tau )^{{\alpha} }},\ 0<{\alpha}
<1.  \label{FD}
\end{equation}
\end{enumerate}

\begin{proposition}
The following hold true for the processes $Y_{\alpha}$ and $L_{\alpha}$:
\begin{enumerate}
\item The Laplace transform is
\begin{equation*}
\E e^{-sY_{{\alpha} }(t)}=\sum_{n=0}^{\infty }\frac{(-st^{{\alpha}
})^{n} }{\Gamma ({\alpha} n+1)}=E_{{\alpha} }(-st^{{\alpha} }),\; s>0
\end{equation*}
and it holds that
\begin{equation*}
\int_{0}^{\infty }e^{-st}f_{\alpha}(t,x)dt=s^{{\alpha} -1}e^{-xs^{{\alpha}
}}, \; s\geq 0.
\end{equation*}

\item Both processes $L_{{\alpha} }(t),t\geq 0$ and $Y_{{\alpha} }(t)$ are
self-similar, i.e.,
\begin{equation*}
\frac{L_{{\alpha} }(at)}{a^{1/{\alpha} }}\overset{d}{=}L_{{\alpha} }(t),\;%
\frac{ Y_{{\alpha} }(at)}{a^{{\alpha} }}\overset{d}{=}Y_{{\alpha} }(t),\; a>0.
\end{equation*}

\item
\begin{equation}  \label{eq:EY}
\E Y_{{\alpha} }(t)=\frac{t^{{\alpha} }}{\Gamma (1+{\alpha} )};\ 
\E [Y_{{\alpha} }(t)]^{\nu }=\frac{\Gamma (\nu +1)}{\Gamma ({\alpha}
\nu +1)} t^{{\alpha} \nu },\ \nu >0;
\end{equation}

\item 
\begin{align}
&\mathrm{Cov}\big(Y_{{\alpha} }(t),Y_{{\alpha} }(s)\big)  \label{COV4} \\
&=\frac{1}{\Gamma (1+{\alpha} )\Gamma ({\alpha} )}\int_{0}^{\min (t,s)}\big(%
(t-\tau )^{{\alpha} }+(s-\tau )^{{\alpha} }\big) \tau ^{{\alpha} -1}d\tau -%
\frac{ (st)^{{\alpha} }}{\Gamma ^{2}(1+{\alpha} )}.  \notag\\
&=\frac{1}{\Gamma (1+{\alpha} )\Gamma ({\alpha} )}[t^{2\alpha}B(s/t;\alpha,1+\alpha) + s^{2\alpha}B(\alpha, 1+\alpha)] -%
\frac{ (st)^{{\alpha} }}{\Gamma ^{2}(1+{\alpha} )},\; s<t,  \notag
\end{align}
where $B(a,b)$ ($B(z;a,b)$, resp.) is the Beta function (incomplete Beta function, resp.).

\item The following asymptotic expansions hold true:
\begin{enumerate}
\item For fixed $s$ and large $t$, it is not difficult to show that
\begin{align}\label{AsympCov}
\mathrm{Cov}\big(Y_{{\alpha} }(t),Y_{{\alpha} }(s)\big) = \frac{-\alpha s^{\alpha+1}t^{\alpha-1}}{\Gamma(\alpha)\Gamma(2+\alpha)}+\cdots + s^{2\alpha}B(1+\alpha, \alpha).
\end{align}
Further, 
\begin{align*}
\mathrm{Var}\big(Y_{{\alpha} }(t)\big) = t^{2\alpha}\left( \frac{2}{\Gamma(1+2\alpha)} - \frac{1}{\Gamma^2(1+\alpha)}\right) = d({\alpha})\,t^{2\alpha}\;\mathrm{(say)}.
\end{align*}
\item For fixed $s$, it follows
\begin{align}\label{AsympCorr}
\mathrm{Cor}\big(Y_{{\alpha} }(t),Y_{{\alpha} }(s)\big) \sim \frac{B(1+\alpha,\alpha)}{d(\alpha)}\frac{s^{\alpha}}{t^{\alpha}},\;\mathrm{as}\;t\rightarrow \infty.
\end{align}
\end{enumerate}
\end{enumerate}
\end{proposition}

\bigskip
A finite variance stationary process $X(t)$ is said to have \emph{long-range-de\-pen\-dence} (LRD) property if $\sum_{k=1}^{\infty}\mathrm{Cov}\big(X(t), X(t+k)\big) = \infty$. Further, for a non-stationary process $Y(t)$ an equivalent definition is given as follows.
\begin{definition}[LRD for non-stationary process]\label{def:LRD}
Let $s>0$ be fixed and $t > s$. Then the process $Y(t)$ is said to have LRD property if
\begin{equation}\label{LRD}
\mathrm{Cor}(Y(s), Y(t)) \sim c(s)t^{-d},\; \mathrm{as}\; t\rightarrow \infty,
\end{equation}
where the constant $c(s)$ is depending on $s$ and $0<d<1$.

\end{definition}
Thus this last property~\eqref{AsympCorr} can be interpreted as long-range dependence of $Y_\alpha$ in view of~\eqref{LRD}.

\begin{remark}
	There is a (complicated) form of all finite-dimensional distributions of $Y_{\alpha}(t),\ t\geq 0,$ in the form of Laplace transforms, see~\cite{B}.
\end{remark}

\subsection{Simulation of the stable subordinator}

In order to obtain trajectories for the stable subordinator it is necessary to simulate random variables with finite Laplace transform satisfying~\eqref{eq:10}. To this end it is necessary that the ${\alpha}$-stable random
variable $L_{\alpha}(t)$ is spectrally positive, which means $\beta=1$ in the standard parametrization (also type~1 parametrization, cf.\ \cite[Definition~1.1.6 and page~6]{SamorodnitskyTaqqu}) of the ${\alpha}$-stable
random variable 
\begin{equation}
L_{\alpha}(t)\sim S({\alpha},\beta,\gamma,\delta; 1)  \label{eq:stable}
\end{equation}
with characteristic function (Fourier-transform) 
\begin{equation*}
\E e^{i s L_{\alpha}(t)}=\exp\left(-\gamma^{\alpha}|s|^{\alpha}\big(%
1-i \beta\tan\frac{\pi{\alpha}}{2}\sign(s)\big)+i\delta s\right). \qquad{%
\alpha}\not=1.
\end{equation*}
This expression is also obtained by substituting $-i s$ in the Laplace
transform 
\begin{equation*}
\E e^{-s L_{\alpha}(t)}=\exp\left(-\frac{\gamma^{\alpha}}{\cos\frac{%
\pi{\alpha}}{2}}s^{\alpha}-\delta s\right).
\end{equation*}
Comparing the latter with~\eqref{eq:10} reveals that the parameters of the ${%
\alpha}$-stable random variable $L_{\alpha}(t)$, ${\alpha}\in(0,1)$, in the
parametrization~\eqref{eq:stable} are 
\begin{equation*}
\beta=1,\ \gamma^{\alpha}=t\cdot\cos\frac{\pi{\alpha}}{2} \text{ and }
\delta=0.
\end{equation*}

\begin{figure}[h!]
\centering
\subfloat[Stable subordinator $L_\alpha(t)$]{\includegraphics[width= 0.5\textwidth, height = 0.4\textwidth]{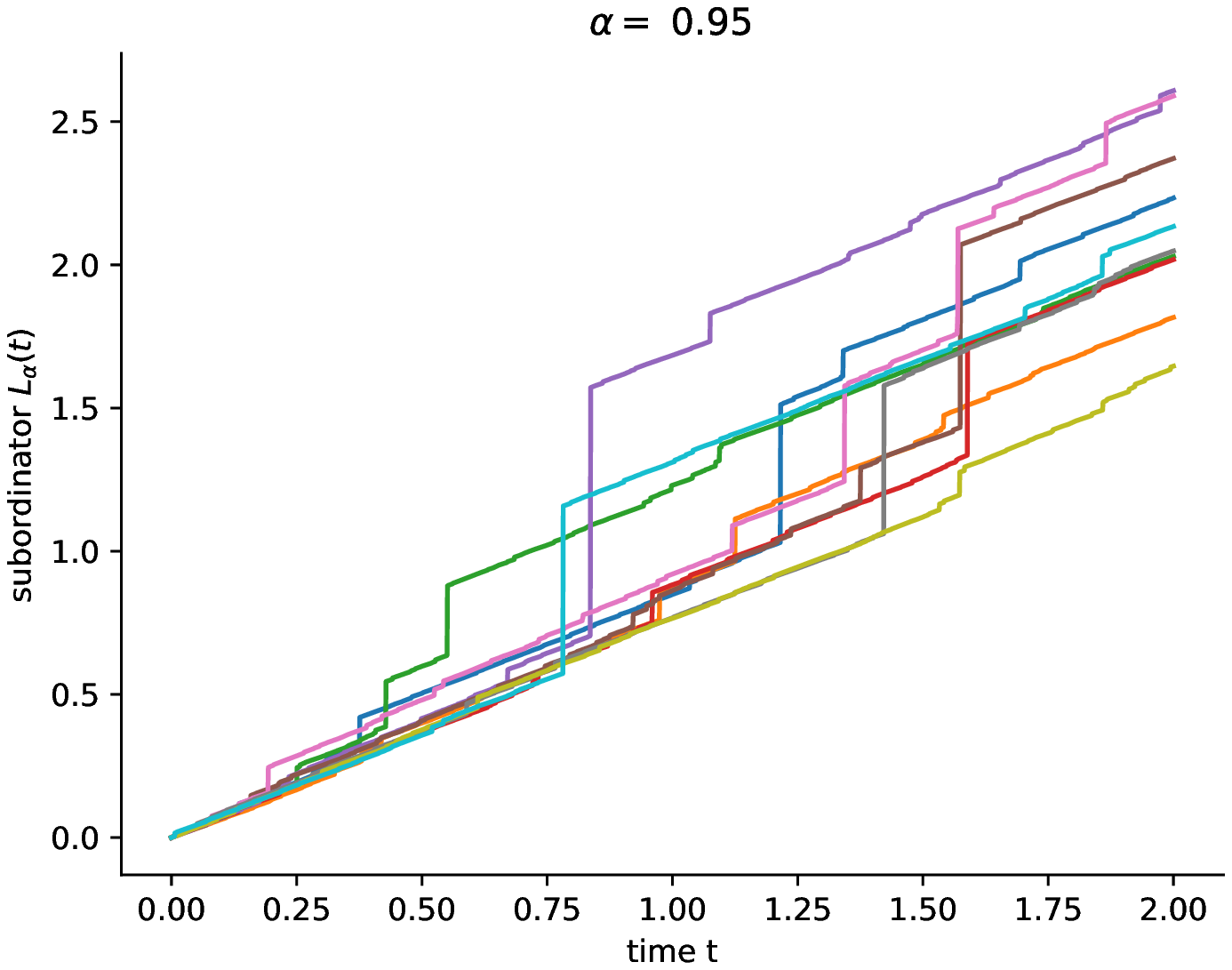}}
\subfloat[Inverse stable subordinator $Y_\alpha(t)$]{\includegraphics[width= 0.5\textwidth, height = 0.4\textwidth]{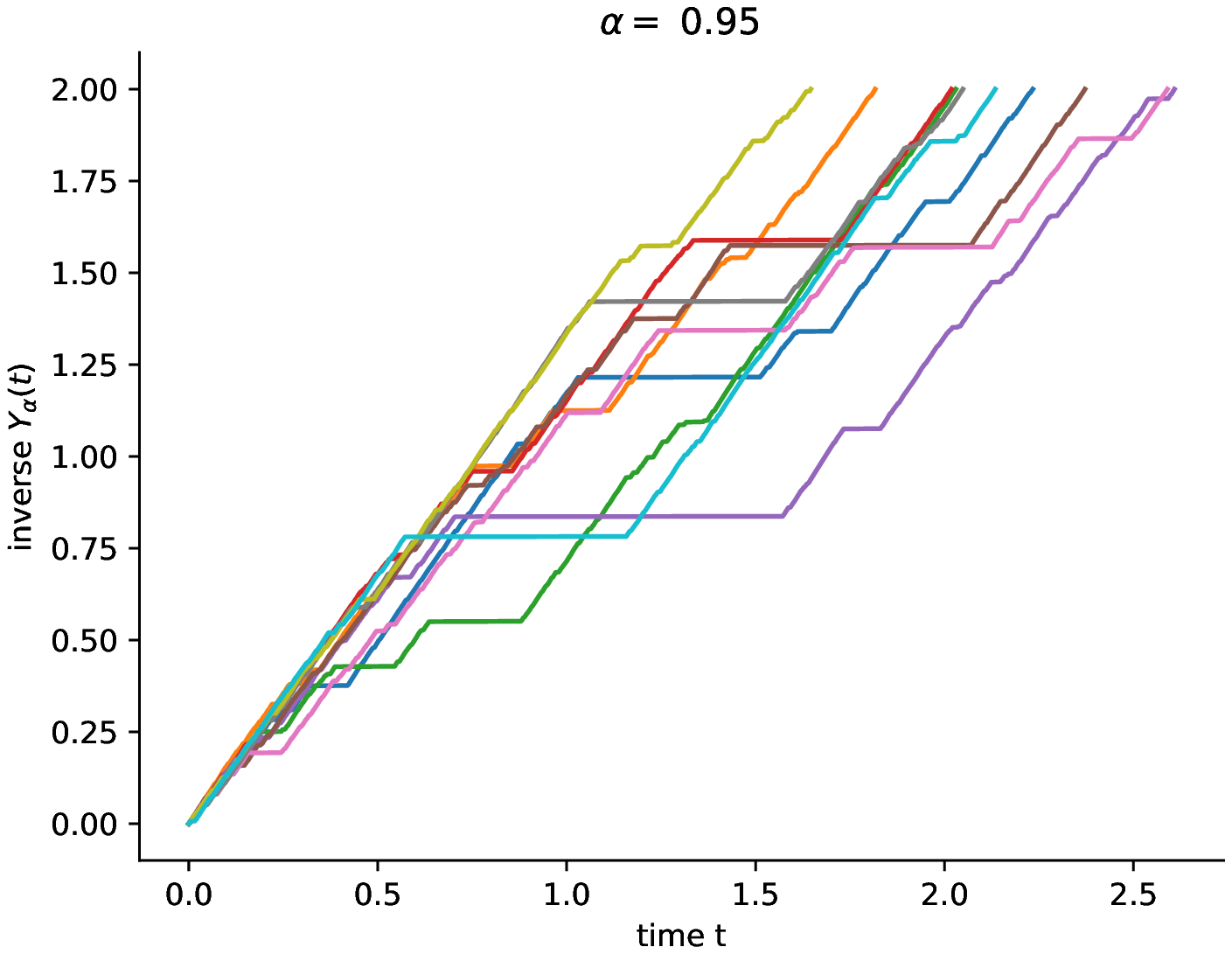}}   
\caption{Sample trajectories of stable and inverse stable subordinators with $\alpha=95\%$.}
\label{fig:1b}
\end{figure}



\section{Classical fractional Poisson processes}
\label{sec:3}

\bigskip The first definition of the fractional Poisson process (FPP) $N_{{\alpha} }=\{N_{{\alpha} }(t),\
t\ge 0\}$ is given by Mainardi et al.~\cite{MGS} (see also \cite{MGV1}) as a
renewal process with Mittag-Leffler waiting times between the events 
\begin{align}
N_{{\alpha} }(t)&=\max \left\{ n\colon T_{1}+\dots+T_{n}\leq t\right\}
=\sum_{j=1}^{\infty }\mathrm{I}\left\{ T_{1}+\dots+T_{j}\leq t\right\} \nonumber\\
&=\sum_{j=1}^{\infty }\mathrm{I}\left\{ U_{j}\leq G_{{\alpha} }(t)\right\}
,\quad t\geq 0,  \label{DF1}
\end{align}
where $\left\{ T_{j}\right\}$, $j=1,2,\dots$, are iid random variables with the
strictly monotone Mittag-Leffler distribution function 
\begin{equation}\label{eq:MittagLeffler}
	F_{{\alpha} }(t)=\mathrm{P}\left( T_{j}\leq t\right) =1-E_\alpha(-\lambda t^\alpha),\ t\geq 0,\ 0<{\alpha} <1,\; j=1,2,\ldots,
\end{equation}
and 
\begin{equation*}
G_{{\alpha} }(t)=\mathrm{P}\left( T_{1}+\dots+T_{k}\leq t\right)
=\int_{0}^{t}h^{(k)}(x)dx.
\end{equation*}
\begin{figure}[ht!]
	\begin{center}
		\includegraphics[width= 0.50\textwidth]{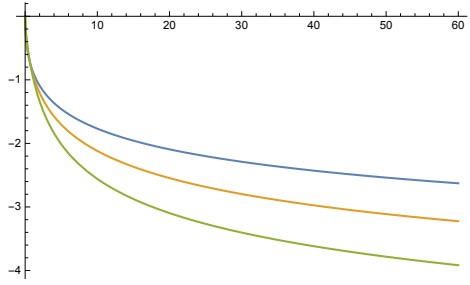}
		\caption{\label{fig:MittagLeffler}The survival function~\eqref{eq:MittagLeffler} for the fractional Poisson process for $\alpha=50\%$, $60\%$ and $70\%$. Compare with real data in Figure~\ref{fig:1}.}
	\end{center}
\end{figure}
Here we denote an indicator as $\mathrm{I}\left\{ \cdot \right\} $ and $%
U_j$, $j=1,2,\dots$, are iid uniformly distributed on $[0,1]$ random
variables. $h^{(j)}(\cdot)$ is the pdf of $j$th convolution of the
Mittag-Leffler distributions which is known as the generalized Erlang
distribution and it is of the form 
\begin{align*}
h^{(k)}(x)&={\alpha} \lambda ^{k}\frac{x^{k{\alpha} -1}}{(k-1)!}E_{{\alpha}
}^{(k)}(-\lambda x^{{\alpha} }) \\
&=\lambda ^{k}x^{{\alpha} k-1}E_{{\alpha} ,{\alpha} k+1}^{k}(-\lambda x^{{%
\alpha} }),\quad {\alpha} \in (0,1),\ x>0,
\end{align*}
where the three-parametric Generalized Mittag-Leffler function is defined as (cf.~\eqref{ML1} and~\cite{HMS}) 
\begin{equation}
E_{{\alpha} ,\beta }^{\gamma }(z)=\sum_{j=0}^{\infty }\frac{(\gamma
)_{j}z^{j} }{j!\,\Gamma ({\alpha} j+\beta )},\quad {\alpha} >0,\ \beta >0,\
\gamma >0,\ z\in \mathbb{C},  \label{ML3}
\end{equation}
where $(\gamma)_j:= \gamma(\gamma+1)\dots(\gamma+j-1)=j!\, \binom{\gamma+j-1}{j%
}=\frac{\Gamma(\gamma+j)}{\Gamma(\gamma)}$ is the rising factorial
(sometimes also called Pochhammer function).

Note that $S(t)\stackrel{d}=(t/L_{{\alpha} }(1))^{{\alpha} }$ and 
\begin{equation*}
\mathrm{P}\left( T_{j}>t\right) =\E e^{-\lambda S(t)},\quad t\geq 0.
\end{equation*}
Meershaert et al.\ \cite{MNV} find the stochastic representation
for FPP
\begin{equation*}
N_{\alpha} (t)=N\big(Y_{{\alpha} }(t)\big),\quad t\geq 0,\ {\alpha} \in
(0,1),
\end{equation*}
where $N=\{N(t),\ t\geq 0\},$ is the classical homogeneous Poisson process
with parameter $\lambda >0,$ which is independent of the inverse stable
subordinator $Y_{{\alpha} }.$

\bigskip One can compute the following expression for the one-dimensional
distribution of FPP: 
\begin{align}  \label{eq:26}
\mathrm{P}\big( N_{{\alpha} }(t)=k\big)&=p_{k}^{({\alpha}%
)}(t)=\int_{0}^{\infty }\frac{e^{-\lambda x}(\lambda t)^{k}}{k!}f_{{\alpha}
}(t,x)dx  \notag \\
&=\frac{t\lambda ^{k}}{{\alpha} k!}\int_{0}^{\infty }e^{-\lambda x}x^{k-1-%
\frac{1}{{\alpha} }}g_{{\alpha} }(tx^{-\frac{1}{{\alpha} }})dx  \notag \\
& =\frac{(\lambda t^{{\alpha} })^{k}}{k!}\mathop{\displaystyle \sum}
\limits_{j=0}^{\infty }\frac{(k+j)!}{j!}\frac{(-\lambda t^{{%
\alpha} })^{j}}{ \Gamma ({\alpha} (j+k)+1)}=\frac{\left( \lambda t^{{\alpha}
}\right) ^{k}}{k!} E_{{\alpha} }^{(k)}(-\lambda t^{{\alpha} }) \\
&=(\lambda t^{{\alpha} })^{k}E_{{\alpha} ,{\alpha} k+1}^{k+1}(-\lambda t^{{%
\alpha} }),\quad k=0,1,2\dots,t\geq 0,\ 0<{\alpha} <1,  \notag
\end{align}%
where $E_{{\alpha} }(z)$ is the Mittag-Leffler function~\eqref{ML1}
evaluated at $z=-\lambda t^{{\alpha} }$, and $E_{{\alpha} }^{(k)}(z)$ is the 
$k$th derivative of $E_{{\alpha} }(z)$ evaluated at $z=-\lambda t^{{\alpha}
}.$ Further, $E_{{\alpha} ,\beta }^{\gamma }(z)$ is the Generalized Mittag-Leffler
function~\eqref{ML3} evaluated at $z=-\lambda t^{{\alpha} }.$

Finally, Beghin and Orsingher (cf.\ \cite{B01}, \cite{BO2}) show that the
marginal distribution of FPP satisfies the system of fractional
differential-difference equations%
\begin{equation*}
\mathrm{D}_{t}^{{\alpha} }p_{k}^{({\alpha} )}(t)=-\lambda (p_{k}^{({\alpha}%
)}(t)-p_{k-1}^{({\alpha} )}(t)),\quad k=0,1,2,\dots
\end{equation*}%
with initial conditions $p_{0}^{({\alpha} )}(0)= 1$, $p_{k}^{({\alpha}%
)}(0)= 0$, $k\geq 1$, where $\mathrm{D}_{t}^{{\alpha} }$ is the fractional
Caputo-Djrbashian derivative~\eqref{FD}.

\begin{remark}[Expectation and variance]
Note that 
\begin{equation}  \label{EXP}
\E N_{{\alpha} }(t)=\int_{0}^{\infty }[\E N(u)]f_{{\alpha}%
}(t,u)du=\frac{\lambda t^{\alpha}}{\Gamma (1+{\alpha})},
\end{equation}%
where $f_{{\alpha} }(t,u)$ is given by~\eqref{D1}; Leonenko et al.~\cite%
{LMS}) show that 
\begin{align}
\mathrm{Cov}(N_{{\alpha} }(t),N_{{\alpha} }(s))& =\int_{0}^{\infty
}\int_{0}^{\infty }[\mathrm{Var}N(1)]\min (u,v)H_{t,s}(du,dv)  \label{COV} \\
& \quad +(\E N(1))^{2}\mathrm{Cov}(Y_{{\alpha} }(t),Y_{{\alpha} }(s))
\notag \\
& =\frac{\lambda (\min (t,s))^{{\alpha} }}{\Gamma (1+{\alpha} )}+\lambda
^{2} \mathrm{Cov}(Y_{{\alpha} }(t),Y_{{\alpha} }(s)),  \notag
\end{align}%
where $\mathrm{Cov}(Y_{{\alpha} }(t),Y_{{\alpha} }(s))$ is given in~\eqref{COV4}, and $H_{t,s}(u,v)=\mathrm{P}\big(Y_{{\alpha} }(t)<u,Y_{{\alpha}
}(s)<v\big).$ In particular 
\begin{align*}
\mathrm{Var}N_{{\alpha} }(t) &=\lambda ^{2}t^{2{\alpha} }\left(\frac{2}{%
\Gamma (1+2{\alpha} )}-\frac{1}{\Gamma ^{2}(1+{\alpha} )}\right)+\frac{%
\lambda t^{{\alpha} }}{ \Gamma (1+{\alpha} )} \\
&=\frac{\lambda ^{2}t^{2{\alpha} }}{\Gamma ^{2}(1+{\alpha} )}\left(\frac{{%
\alpha} \Gamma ({\alpha} )}{\Gamma (2{\alpha} )}-1\right)+\frac{\lambda t^{{%
\alpha} }}{\Gamma (1+{\alpha} )},\quad t\geq 0.
\end{align*}%
For $d=1$ the definition of Hurst index for renewal processes is discussed
in~\cite{DAL}. For the FPP it is equal 
\begin{equation*}
H=\inf \left\{ \beta \colon\lim \sup_{t\rightarrow \infty }\frac{\mathrm{Var}N_{{%
\alpha} }(t)}{t^{\beta }}<\infty \right\} ,
\end{equation*}%
thus $H=2{\alpha} \in (0,2).$
\end{remark}

Finally, Leonenko et al.~\cite{LST, LST1} introduced a fractional non-homogeneous
Poisson process (FNPP) with an intecity function $\lambda (t)\geq 0$ as 
\begin{align}
NN_{\alpha}(t)&:=N\Big(\Lambda\big(Y_{{\alpha} }(t)\big)\Big),\quad t\geq
0,\ {\alpha} \in (0,1)\text{ and}  \notag \\
\Lambda(t)&:=\int_0^t\lambda (s)ds,
\label{eq:lambda}
\end{align}%
where $N=\{N(t),t\geq 0\}$ is the classical homogeneous Poisson process with
parameter $1$, which is independent of the inverse stable subordinator $Y_{{%
\alpha} }.$ 
 Note that
\[
\mathbb{P}(NN_{\alpha}(t) = k) = \int_{0}^{\infty} e^{-\Lambda(u)} \frac{\Lambda(u)^k}{k!}f_{\alpha}(t,u)du,\;k=0,1,2\ldots,
\]
where $f_{\alpha}$ is given by \eqref{D1}. Alternatively (cf.~\cite{LST1}),
\[
NN_{\alpha}(t) = \sum_{n=0}^{\infty}n\, \mathds{I}{\{L_{\alpha}(\zeta_n) \leq t < L_{\alpha}(\zeta_{n+1})\}},
\]
where $\xi_1, \xi_2, \ldots, \xi_n$ is a sequence of non-negative iid random variables such that $\mathbb{P}(\xi_1 \leq x) = 1-e^{-\Lambda(x)},\;x\geq 0;$ $\zeta_n' = \max\{\xi_1,\xi_2,\ldots,\xi_n\}$ and $\zeta_n = \zeta_{\chi_n}'$, where $\chi_n = \inf\{k\colon \zeta_k' > \zeta_{\chi_{n-1}}'\},\;n=2,3,\ldots,$ with $\chi_1=1.$ The resulting sequence $\zeta_1, \zeta_2, \ldots, \zeta_n$ is strictly increasing, since it is obtained from the non-decreasing sequence $\zeta_1', \zeta_2',\ldots,\zeta_n'$ by omitting all repeating elements.

%

\section{Fractional risk processes}\label{sec:4}
The risk process is of fundamental importance in insurance. Based on the inverse subordinator process $Y(t)$ (cf.~\eqref{LS3}) we extend the classical risk process (also known as surplus process) and consider
\begin{equation}\label{eq:17}
	R_{\alpha}(t):=u+\mu\,\lambda\,(1+\rho)\,Y_{{\alpha }}(t)-\sum_{i=1}^{N_{{%
	\alpha }}(t)}X_{i},\quad t\geq 0.
\end{equation}%
Here, $u>0$ is the initial capital relative to the number of claims per time
unit (the Poisson parameter $\lambda $) and the iid variables $\{X_{i}\}$, $i\geq 1$ with mean $\mu >0$ model claim sizes. $N_{{\alpha }}(t)$, $t\geq 0,$
is an independent of $\{X_{i}\}$, $i\geq 1$, fractional Poisson process. The parameter $\rho \geq 0$ is the safety loading; we demonstrate in Proposition~\ref{prop:Martingale} below that the risk process~\eqref{eq:17} satisfies the \emph{net profit condition} iff $\rho >0$ (cf.\ Mikosch~\cite[Section~4.1]{Mikosch2009}),
which is economically that the company will not necessarily go bankrupt iff $\rho>0$.

\begin{remark}\label{rem:Risk}
	Note that $Y_1(t)=t$ for $\alpha=1$ so that the model~\eqref{eq:17} extends the classical ruin process considered in risk theory.
\end{remark}
	
It is essential in~\eqref{eq:17} to observe that the counting process $N_\alpha(t)$ and the payment process $Y_\alpha(t)$ follow the same time scale. These coordinated, or harmonized clocks are essential, as otherwise the model would over-predict too many claims (too many premiums, respectively). As well, different clocks for these two processes violated the profit condition (cf.\ Proposition~\ref{prop:Martingale} below).

The process for the non-homogeneous analogue of~\eqref{eq:17} is 
\begin{equation*}
	R_{\alpha }(t)=u+\mu (1+\rho )\Lambda \big(Y_{{\alpha }}(t)\big)%
	-\sum_{i=1}^{NN_{{\alpha }}(t)}X_{i},\ t\geq 0.
\end{equation*}

\begin{remark}[Consistency with equivalence principle]
	A main motivation for considering the risk process~\eqref{eq:17} comes from the fact that the stochastic processes 
	\begin{equation*}
		\sum_{i=1}^{N_{{\alpha} }(t)}X_i-\mu\,\lambda\,Y_{{\alpha} }(t),\quad t\geq 0
	\end{equation*}%
	and its non-homogeneous analogon
	\begin{equation*}
		\sum_{i=1}^{NN_\alpha(t)}X_i-\mu\,\Lambda\big(Y_{{\alpha}}(t)\big),\quad t\geq
	0
	\end{equation*}%
	are martingales with respect to natural filtrations (cf.\ Proposition~\ref{prop:Martingale} below).
	
	It follows from this observation that the time change imposed by $Y_{\alpha}$
	does not affect or violate the \emph{net} or \emph{equivalence principle}.
	 Further, the part of the risk process corresponding to the premium (i.e., $\mu\,\lambda\,	Y_\alpha(t)$ without $\rho$, or by setting $\rho=0$) is the fair premium
	of the remaining claims process even under the fractional Poisson process.
\end{remark}

\begin{proposition}\label{prop:Martingale}
The fractional risk process $R_{\alpha}(t)$ introduced in~\eqref{eq:17} is a submartingale (martingale, supermartingale, resp.) for $\rho>0$ ($\rho=0$, $\rho<0$, resp.) with respect to the natural filtration.
\end{proposition}

\begin{proof}
Note that the compensated FPP $N_{\alpha}(t)-\lambda Y_{\alpha}(t)$ is a martingale with respect to the filtration $\mathcal{F}_t = \sigma(N_{\alpha}(s),\; s \leq t) \vee \sigma(Y_{\alpha}(s),\; s \geq 0)$, cf.\ \cite{ALM}. We have
\begin{align*}
\E[&R_{\alpha}(t) - R_{\alpha}(s)\mid\mathcal{F}_s]\\ &= \E\left[\lambda\mu(1+\rho) Y_{\alpha}(t) - \sum_{i=1}^{N_{\alpha}(t)}X_i - \left(\lambda \mu(1+\rho) Y_{\alpha}(s) - \sum_{i = 1}^{N_{\alpha}(s)}X_i \right)\bigg|\mathcal{F}_s\right]\\
& =\E\left[\left(\lambda\mu(1+\rho) (Y_{\alpha}(t)-Y_{\alpha}(s)) - \sum_{i = N_{\alpha}(s)+1}^{N_{\alpha}(t)} X_i\right)\bigg|\mathcal{F}_s\right]\\
& =\E\left[\left(\lambda \mu(1+\rho) (Y_{\alpha}(t)- Y_{\alpha}(s))\right)\bigg|\mathcal{F}_s\right] - \E\left[\E\left[\sum_{i = N_{\alpha}(s)+1}^{N_{\alpha}(t)}X_i\Bigg|\mathcal{F}_t\right]\bigg|\mathcal{F}_s\right]\\
& =\E\left[\left(\lambda\mu(1+\rho) (Y_{\alpha}(t)- Y_{\alpha}(s))\right)\bigg| \mathcal{F}_s\right] - \E\left[\left(N_{\alpha}(t)- N_{\alpha}(s)\right)\mu\bigg|\mathcal{F}_s\right] \\
&= \mu \E\left[\left(\lambda(1+\rho)(Y_{\alpha}(t) - Y_{\alpha}(s))\right) - \left(N_{\alpha}(t)-N_{\alpha}(s)\right)\bigg|\mathcal{F}_s\right]\\
&= - \mu\E\left[\left(N_{\alpha}(t)-\lambda Y_{\alpha}(t)\right)-\left(N_{\alpha}(s)-\lambda Y_{\alpha}(s)\right)\bigg|\mathcal{F}_s\right] + \lambda\mu\rho\E\left[Y_{\alpha}(t) - Y_{\alpha}(s)\Bigg| \mathcal{F}_s\right]\\
&=\lambda\mu\rho\E\left[Y_{\alpha}(t) - Y_{\alpha}(s)\bigg| \mathcal{F}_s\right],
\end{align*}
since the compensated FPP is a martingale.
Thus 
\[
 \E\big[R_{\alpha}(t) - R_{\alpha}(s)\mid\mathcal{F}_s\big]=
 	\begin{cases} 
      > 0 & {\rm if}\;\rho>0, \\
      =  0 & { \rm if}\;\rho = 0, \\
      <0  & { \rm if}\;\rho < 0.
	\end{cases}\]
This completes the proof.
\end{proof}

\begin{remark}[Marginal moments of the risk process]
The expectation and the covariance function of the
process loss $X_{1}+\dots+X_{N_{\alpha }(t)}$ in~\eqref{eq:17} can be given
as%
\begin{equation}\label{eq:19}
\E\sum_{i=1}^{N_{{\alpha }}(t)}X_{i}=\frac{\lambda t^{\alpha }}{%
\Gamma (1+\alpha )}\E X_{i}
\end{equation}%
and, by employing~\cite{LMS}, 
\begin{align*}
	\mathrm{Cov}\left(\sum_{i=1}^{N_{{\alpha }}(t)}X_{i},\sum_{i=1}^{N_{{\alpha }%
	}(s)}X_{i}\right)=&\lambda t\E X_{i}^{2}\frac{\lambda (\min
	(t,s))^{\alpha }}{\Gamma (1+\alpha )}\\
	&+\left[\lambda \E X_{i}^{2}+\lambda ^{2}(\E X_{i})^{2}\right]\mathrm{Cov}%
	(Y_{{\alpha }}(t),Y_{{\alpha }}(s)),
\end{align*}%
where $\mathrm{Cov}\big(Y_{{\alpha }}(t),Y_{{\alpha }}(s)\big)$ is given in~\eqref{COV4}.
\end{remark}

\subsection{A variant of the surplus process}
A seemingly simplified version of the surplus process~\eqref{eq:17} is obtained by replacing the processes $Y_{{\alpha }%
}(t) $ or $\Lambda\big(Y_{{\alpha }}(t)\big)$ by their expectations (explicitly given in~\eqref{eq:EY}) and to consider
\begin{equation}\label{eq:18}
\tilde R_\alpha(t)=u+\mu\,\lambda\,(1+\rho)\frac{t^{{\alpha 
}}}{\Gamma (1+{\alpha })}-\sum_{i=1}^{N_{{\alpha }}(t)}X_{i},\quad t\geq 0,
\end{equation}
or, more generally,
\begin{equation*}
\tilde R_\alpha(t)=u+\mu \E\Lambda\big(Y_{\alpha}(t)\big)-\sum_{i=1}^{NN_{\alpha}(t)}X_i,\quad t\geq 0
\end{equation*}
for the non-homogeneous process.
As above, $u$ is the initial capital, $\mu$ is the constant premium rate and the
sequence of iid random variables $(X_i)_{i\geq 1}$ is independent of FFP $N_{{\alpha} }(t)$.
The net profit condition, as formulated by Mikosch \cite{Mikosch2009}, involves the expectation only. For this reason, the adapted surplus process~\eqref{eq:18} satisfies Mikosch's net profit condition as well.

For $\alpha=1$ we have that $Y_1(t)=t$ so that the simplified process~\eqref{eq:18} coincides with the classical surplus process considered in insurance 
 (cf.\ Remark~\ref{rem:Risk}). However, the simplified risk process $\tilde R_\alpha(t)$, in general, is \emph{not} a martingale unless $\alpha=1$ and $\rho=0$, while $R_\alpha(t)$ is martingale for any $\alpha$ and $\rho=0$ by Proposition~\ref{prop:Martingale}.

\begin{remark}[Time shift]\label{rem:time}
The formula~\eqref{eq:18} reveals an important property of the fractional
Poisson process and the risk process $R_{\alpha}(t)$. Indeed, for small
times $t$, there are more claims to be expected under the FPP regime, as 
\begin{equation*}
\frac{t^{\alpha}}{\Gamma(1+{\alpha})} >t\quad\text{for }t\text{ small}.
\end{equation*}

The inequality reverses for later times. This means that the premium income rate decays later. The time change imposed by FPP postpones claims to later
times --- another feature of the fractional Poisson process and a consequence of the martingale property.

However, the premium income is $\lambda\mu t$ in a real world situation. For this we conclude that the FPP can serve as a stress test for insurance companies
within a small, upcoming time horizon.
\end{remark}



\subsection{Long range dependence of risk process $R_{\alpha}(t)$}
\label{sec:42}
In this section, we discuss the long range dependency property (LRD property, see Definition~\ref{def:LRD}) of the risk process $R_{\alpha}(t)$.
\begin{proposition}
The covariance structure of risk process $R_{\alpha}(t)$ is given by
\begin{equation}\label{CovRiskProcess}
\mathrm{Cov}\big(R_{\alpha}(t), R_{\alpha}(s)\big)= \mu^2\lambda^2 \rho^2 \mathrm{Cov}(Y_{\alpha}(t), Y_{\alpha}(s))
 +  \frac{\lambda (\E X_i^2) s^{{\alpha} }}{\Gamma (1+{\alpha} )},\; s\leq t.
\end{equation}
Further, the variance of $R_{\alpha}(t)$ is
\begin{equation*}
\mathrm{Var}(R_{\alpha}(t)) = \mu^2\lambda^2 \rho^2 \mathrm{Var}(Y_{\alpha}(t)) + \frac{\lambda (\E X_i^2) t^{{\alpha} }}{\Gamma (1+{\alpha} )}.
\end{equation*}
\end{proposition}

\begin{proof}
For $s\leq t$, we have $N_{\alpha}(s)\leq N_{\alpha}(t)$.
\begin{align*} 
\mathrm{Cov}&\left(\sum_{i=1}^{N_{{\alpha }}(t)}X_{i},\sum_{j=1}^{N_{{\alpha }}(s)}X_{j}\right) \\
=& \E  \left(\sum_{i=1}^{\infty}\sum_{j=1}^{\infty} X_i X_j I\{N_{\alpha}(s)\geq j, N_{\alpha}(t)\geq i\}\right) - (\E X_i)^2 \E (N_{\alpha}(t))\E (N_{\alpha}(s))\\
=& \E \left(\sum_{j=1}^{\infty} X_j^2 I\{N_{\alpha}(s)\geq j\}\right) + \E \left(\mathop{\sum\sum}_{i\neq j} X_iX_j I\{N_{\alpha}(s)\geq j, N_{\alpha}(t)\geq i\}\right ) \\
& -(\E X_i)^2 \E (N_{\alpha}(t))\E (N_{\alpha}(s))\\
 =&  \E (X_j^2) \sum_{j=1}^{\infty}\mathrm{P}(N_{\alpha}(s)\geq j) \\
& +(\E X_i)^2 \left(\sum_{i=1}^{\infty}\sum_{j=1}^{\infty} \mathrm{P}\left(N_{\alpha}(s)\geq j, N_{\alpha}(t)\geq i\right ) -\sum_{j=1}^{\infty}\mathrm{P}(N_{\alpha}(s)\geq j)\right) \\
&- (\E X_i)^2 \E (N_{\alpha}(t))\E (N_{\alpha}(s))\\
=&  \E (X_i^2) \E N_{\alpha}(s) +  (\E X_i)^2  \E (N_{\alpha}(s)N_{\alpha}(t))\\
&-(\E X_i)^2 \E N_{\alpha}(s)- (\E X_i)^2 \E (N_{\alpha}(t))\E (N_{\alpha}(s))\\
=&  \E (X_i^2) \E N_{\alpha}(s) +  (\E X_i)^2 \mathrm{Cov}\left(N_{\alpha}(t), N_{\alpha}(s)\right) - (\E X_i)^2 \E N_{\alpha}(s)\\
=& \mathrm{Var}(X_i)\E N_{\alpha}(s) + (\E X_i)^2 \mathrm{Cov}\left(N_{\alpha}(t), N_{\alpha}(s)\right) \\
=& \mathrm{Var}(X_i)\E N_{\alpha}(s) +   (\E X_i)^2 \frac{\lambda s^{{\alpha} }}{\Gamma (1+{\alpha} )}+ \lambda^{2} \mu^2 \mathrm{Cov}(Y_{{\alpha} }(t),Y_{{\alpha} }(s))\\
=& \mathrm{Var}(X_i)\E N_{\alpha}(s) +   (\E X_i)^2 \E N_{\alpha}(s)+ \lambda^{2} \mu^2 \mathrm{Cov}(Y_{{\alpha} }(t),Y_{{\alpha} }(s))\\
=& \E (X_i^2)\E N_{\alpha}(s) + \lambda^{2} \mu^2 \mathrm{Cov}(Y_{{\alpha} }(t),Y_{{\alpha} }(s))
\end{align*}
Further,
\begin{align*}
\mathrm{Cov}\left(Y_{\alpha}(s), \sum_{i=1}^{N_{\alpha}(t)}X_i\right) &= \E \left(Y_{\alpha}(s) \sum_{i=1}^{N_{\alpha}(t)}X_i \right) - \E (Y_{\alpha}(s)) \E \left( \sum_{i=1}^{N_{\alpha}(t)}X_i \right) \\
& = \E \left(\E \left(Y_{\alpha}(s)\sum_{i=1}^{N_{\alpha}(t)}X_i\mid Y_{\alpha}(s),N_{\alpha}(t)\right)\right) \\
& \qquad -\mu\lambda \E (Y_{\alpha}(s)) \E (Y_{\alpha}(t))\\
& = \mu \E (Y_{\alpha}(s)N_{\alpha}(t)) - \mu\lambda \E (Y_{\alpha}(s)) \E (Y_{\alpha}(t))\\
& = \mu\lambda \E (Y_{\alpha}(s)Y_{\alpha}(t)) - \mu\lambda \E (Y_{\alpha}(s)) \E (Y_{\alpha}(t))\\
&= \mu\lambda\mathrm {Cov}((Y_{\alpha}(t), Y_{\alpha}(s))
\end{align*}
Similarly, 
$$\mathrm{Cov}\left(Y_{\alpha}(t), \sum_{i=1}^{N_{\alpha}(s)}X_i\right) = \mu\lambda\mathrm {Cov}((Y_{\alpha}(t), Y_{\alpha}(s)).$$
Finally,
\begin{align*}
\mathrm{Cov}(R_{\alpha}(t), R_{\alpha}(s)) & = \mu^2\lambda^2(1+\rho)^2 \mathrm {Cov}((Y_{\alpha}(t), Y_{\alpha}(s))\\
& \qquad - \mu\lambda(1+\rho)\mathrm{Cov}\left(Y_{\alpha}(s), \sum_{i=1}^{N_{\alpha}(t)}X_i\right)\\
 & \qquad -\mu\lambda(1+\rho)\mathrm{Cov}\left(Y_{\alpha}(t), \sum_{i=1}^{N_{\alpha}(s)}X_i\right) \\
 &\qquad +\mathrm{Cov}\left(\sum_{i=1}^{N_{{\alpha }}(t)}X_{i},\sum_{j=1}^{N_{{\alpha }}(s)}X_{j}\right)\\
 & = \mu^2\lambda^2((1+\rho)^2-2(1+\rho) +1) \mathrm {Cov}((Y_{\alpha}(t), Y_{\alpha}(s))\\
 &\qquad +\E (X_i^2)\E N_{\alpha}(s)\\
 & = \mu^2\lambda^2\rho^2\mathrm {Cov}((Y_{\alpha}(t), Y_{\alpha}(s)) + \E (X_i^2)\E N_{\alpha}(s)\\
 &= \mu^2\lambda^2 \rho^2 \mathrm{Cov}(Y_{\alpha}(t), Y_{\alpha}(s))
 +  \frac{\lambda (\E X_i^2) s^{{\alpha} }}{\Gamma (1+{\alpha} )},
\end{align*}
which completes the proof.
\end{proof}

Next, we show that the risk process $R_{\alpha}(t)$ has the LRD property.
\begin{proposition}
The process $R_{\alpha}(t)$ has LRD property for all $\alpha \in(0,1).$
\end{proposition}

\begin{proof}
Note that
\begin{align*}
\mathrm{Var}(R_{\alpha}(t)) &= \mu^2\lambda^2 \rho^2 \mathrm{Var}(Y_{\alpha}(t)) + \frac{\lambda (\E X_i^2) t^{{\alpha} }}{\Gamma (1+{\alpha} )}\\
&\sim \mu^2\lambda^2\rho^2 d(\alpha) t^{2\alpha}\;\mathrm{as}\;t\rightarrow \infty.
\end{align*}
For fixed $s$ and large $t$ with~\eqref{AsympCov} and~\eqref{CovRiskProcess}, it follows
\begin{align*}
\mathrm{Cor}(R_{\alpha}(t), R_{\alpha}(s)) &= \frac{\mu^2\lambda^2 \rho^2 \mathrm{Cov}(Y_{\alpha}(t), Y_{\alpha}(s))
 +  \frac{\lambda (\E X_i^2) s^{{\alpha} }}{\Gamma (1+{\alpha} )}}{\sqrt{\mathrm{Var}(R_{\alpha}(t))} \sqrt{\mathrm{Var}(R_{\alpha}(s))}}\\
 & \sim \frac{\left[\frac{-\alpha s^{\alpha+1}t^{\alpha-1}}{\Gamma(\alpha)\Gamma(2+\alpha)}+\cdots + s^{2\alpha}B(1+\alpha, \alpha)\right]\mu^2\lambda^2 \rho^2 + \frac{\lambda (\E X_i^2) s^{{\alpha} }}{\Gamma (1+{\alpha} )}}{\sqrt{\mathrm{Var}(R_{\alpha}(s))}\mu \lambda \rho \sqrt{d(\alpha)}t^{\alpha}}\\
 & \sim \frac{\mu^2\lambda^2 \rho^2 B(1+\alpha, \alpha)s^{2\alpha} + \frac{\lambda (\E X_i^2) s^{{\alpha} }}{\Gamma (1+{\alpha} )} }{\mu \lambda \rho \sqrt{d(\alpha)}\sqrt{\mathrm{Var}(R_{\alpha}(s))}}t^{-\alpha}.
\end{align*}
Hence the process $R_{\alpha}(t)$ has LRD property in view of~\eqref{LRD}.
\end{proof}

\subsection{Ruin probabilities}

We are interested in the marginal distributions of the stochastic
processes
\begin{equation}  \label{eq:1}
{\underline R}_{\alpha}(t):=\inf_{0\leq s\leq t}\{R_{\alpha}(s)\}\text{ and }
\psi_{\alpha}(v,t):=\mathrm{P}\{{\underline R}_{\alpha}(t)\leq v\},
\end{equation}%
which generalize the classical ruin probability $\psi (0,t)$ (${\alpha}=1$).

We demonstrate next that the ruin probability for the classical and
fractional Poisson process coincide.

\begin{lemma}
\label{lem:Ruin} It holds that 
\begin{equation*}
	\psi(v):=\lim_{t\to\infty}\psi_{\alpha}(v,t)
\end{equation*}
is independent of ${\alpha}$ and coincides with the classical ruin
probability.
\end{lemma}

\begin{proof}
It holds that $P\big(L_{\alpha}(t)\le 0\big)=0$ (the emphasis being here that $0$
does not concentrate measure, $P(L_{\alpha}(t)=0)=0$). The inverse process $Y_{\alpha}$ is continuous (cf.\ the discussion after~\eqref{D1} and Figure~\ref{fig:1b} for illustration), almost surely strictly increasing without jumps  and hence the range of $Y_{\alpha}(\cdot)$ is $\mathbb{R}_{\ge0}$ almost surely.
The composition $N_{\alpha}(t)=N(Y_{\alpha}(t))$ thus attains every value as
well, as the non-fractional (classical) process $R(\cdot)$.
\end{proof}

\bigskip Note that if there exists%
\begin{equation}  \label{eq:Lambda}
B(\cdot)=\Lambda ^{(-1)}(\cdot),
\end{equation}
then ruin probability for non-homogeneous processes can be reduced to the
ruin probability for homogeneous one via formulae:%
\begin{align*}
&R_\alpha(t)=u+\mu\,\Lambda\big(Y_{{\alpha} }(t)\big)-\sum_{i=1}^{NN_{\alpha}(t)}X_i=RH\big(%
\Lambda ^{(-1)}(t)\big), \\
&RH(t):=u+\mu\,\lambda\,Y_{{\alpha} }(t)-\sum_{i=1}^{N_{{\alpha} }(t)}X_i,\
t\geq 0,
\end{align*}
and hence 
\begin{equation*}
\mathrm{P}\{R(s)\leq v,\ 0\leq s\leq t\}=\mathrm{P}\{RH(s)\leq v,\ 0\leq
s\leq\Lambda (t)\}.
\end{equation*}
We thus have the following corollary.

\begin{corollary}[to Lemma~\protect\ref{lem:Ruin}]
\label{cor:1} The assertions of Lemma~\ref{lem:Ruin} hold even for the
fractional (FPP) and non-homogeneous process, provided that $\Lambda(\cdot)$
is unbounded.
\end{corollary}

\section{Premiums based on the fractional Poisson processes}
\label{sec:5}

Risk measures are designed to quantify the risk associated with a random, $\mathbb R$\nobreakdash-valued random variable $Z\in\mathcal Z$ and they are employed to compute insurance premiums. In what follows we consider risk measures on $\mathcal Z=L^p$ for an appropriate $p\in [1,\infty]$ and use them to express the average capital for a bailout after ruin for the fractional Poisson process and to compare premiums related to the classical and fractional Poisson process.

\begin{definition}[Coherent measures of risk, cf.\ \cite{adeh}]\label{def:Coherent}
	Let $\mathcal Z$ be an appropriate Banach space of random variables. A \emph{risk functional} is a convex mapping $\rho\colon\mathcal Z\to\mathbb{R}$,
	where 
	\begin{enumerate}[nolistsep, noitemsep]
		\item \label{lab:i}$\rho(\lambda Z)= \lambda\rho(Z)$ whenever $\lambda>0$ and $Z\in\mathcal Z$ (positive homogeneity),
		\item $\rho(Z+c)= \rho(Z)+c$ whenever $c\in\mathbb{R}$ (cash invariance),
		\item $\rho(Z)\le\rho(Y)$ for $Z\le Y$ a.s.\ (monotonicity), and
		\item \label{lab:iv}$\rho(Z+Y)\le\rho(Z)+\rho(Y)$ (subadditivity)
	\end{enumerate}
	for every $Z$, $Y\in\mathcal{Z}$.
\end{definition}

The most prominent risk functional is the \emph{Average Value-at-Risk}, as any other, law-invariant risk functional satisfying \ref{lab:i}--\ref{lab:iv} above can be expressed by $\AVaR$s (see Kusuoka~\cite{Kusuoka}):

\begin{definition}[Average Value-at-Risk]
	The \emph{Average Value-at-Risk} at risk level ${\beta}\in (0,1)$ is the functional 
	\begin{equation*}
	\AVaR_{\beta}(Z):=\frac{1}{1-\beta}\int_\beta^1 F_Z^{-1}(u)\mathrm{d}u,,\qquad Z\in\mathcal Z,
	\end{equation*}
	where $F_Z^{-1}(u):=\VaR_{u}(Z):=\inf\left\{z\in\mathbb{R}\colon P(Z\le
	z)\ge u\right\}$ is the \emph{Value-at-Risk} (quantile function) at level $u$ ($u\in [0,1]$).
\end{definition}

An equivalent expression for the Average Value-at-Risk, which will be useful later, is the formula
\begin{equation}\label{eq:AVaR2}
	\AVaR_{\beta}(Z)= \E\left[Z\mid Z\ge \VaR_{\beta}(Z)\right]
\end{equation}
derived in \cite{PR}.

\subsection{Bailout after ruin}

The probability of ruin introduced in~\eqref{eq:1} describes the probability of bankruptcy of an insurance company based on the surplus process. The quantity does not reveal any information of how much capital is required to recover the company after ruin. This can be
accomplished by considering the random variable 
\begin{equation*}
{\underline R}_{\alpha}:=\inf_{t\ge 0} {\underline R}_{\alpha}(t)
\end{equation*}
in an infinite time horizon with corresponding distribution (ruin probability) 
\begin{equation}  \label{eq:psi}
\psi(v):=P\{{\underline R}_{\alpha}\le v\}.
\end{equation}
In this setting, the average capital required to recover the company from bank\-ruptcy (i.e, conditional on bankruptcy) is 
\begin{equation}  \label{eq:kappa}
\kappa_\alpha(v):=-\E\left({\underline R}_{\alpha}|\ \underline R_{\alpha}\le  v\right).
\end{equation}

By involving the ruin probability~\eqref{eq:psi} the average capital required to recover from bankruptcy thus is 
\begin{equation}  \label{eq:AVaR}
\kappa_\alpha(v)=-\E \left(\underline R_\alpha\mid \underline R_\alpha\le v\right)=\AVaR_{1-\psi(v)}\left(-\underline R_\alpha\right),
\end{equation}
where the risk level involves $\psi(v)$, the cdf~\eqref{eq:psi}.
We shall discuss this quantity in the limiting case corresponding to large companies, for which claims occur instantaneously. 

\medskip

It is well-known that the compensated Poisson process converges to the Wiener process weakly (in distribution, cf.\ \cite{Cont2004}), that is 
\begin{equation}
\left(\frac{N_t-\lambda t}{\lambda}\right)_{t\in[0,T]}\xrightarrow[\lambda%
\to\infty]{}(W_t)_{t\in[0,T]}
\end{equation}
and thus more generally for the risk process (based on Wald's formula)
\begin{align}
\frac{1}{\lambda\cdot\E X^2}& \left(\lambda
u+\lambda\mu(1+\rho)t-\sum_{j=1}^{N_t}X_j\right)_{t\in[0,T]}  \notag \\
&\xrightarrow[\lambda\to\infty]{}\left(\frac{u}{\E X^2}+W_t+\frac{%
\mu\rho}{\E X^2}t\right)_{t\in[0,T]},  \label{eq:34}
\end{align}
which is a Brownian motion with drift 
\begin{equation}  \label{eq:delta}
\delta:=\rho\frac{\mu}{\E X^2}
\end{equation}
started at 
\begin{equation*}
u_0:=\frac{u}{\E X^2}.
\end{equation*}
We shall refer to the process~\eqref{eq:34} as the \emph{limiting case} of
the risk process.

\begin{lemma}\label{lem:16}
	The average capital after ruin defined in~\eqref{eq:kappa} can be recovered
	by 
	\begin{equation*}
		\kappa_\alpha(v)=\frac{1}{\psi(v)}\int_v^\infty \psi(-u)\mathrm{d }u;
	\end{equation*}
	in particular,
	\begin{equation*}
	\kappa:=\kappa_\alpha(0)=\frac{1}{\psi(0)}\int_0^\infty \psi(-u)\mathrm{d }u.
	\end{equation*}
\end{lemma}

\begin{proof}
	The proof is an immediate consequence of~\eqref{eq:AVaR} and the fact that 
	\begin{equation*}
	\E Z=\int_0^\infty P(Z>z)\mathrm{d }z,
	\end{equation*}
	valid for every non-negative random variable, $Z\ge 0$ a.s.\ and applied to $Z:=\underline R_\alpha$ here.
\end{proof}

\begin{remark}[Relation to the fractional Poisson process]
	It is an essential consequence of Lemma~\ref{lem:Ruin} that the average capital to recover from ruin is \emph{independent} of $\alpha$ in a regime driven by the fractional Poisson process.
\end{remark}

\begin{theorem}[The limiting case]
\label{thm:8} The capital requirement $\kappa(v)$
defined in~\eqref{eq:kappa} is independent of $u$ and it holds that 
\begin{equation}  \label{eq:Capital}
\kappa(v)=-v+\frac {\E X^2}{2\rho\mu}.
\end{equation}
\end{theorem}

\begin{proof}
We infer from Borodin and Salminen~\cite[Eq.~(1.2.4)]{Borodin} that 
\begin{equation}  \label{eq:exp}
P\left\{\inf_{t>0}W_t+u_0 +\delta t\le v\right\}=
P_{u_0}\left\{\inf_{t>0}W_t+\delta t\le v\right\}=e^{2\delta(v-u_0)}
\end{equation}
so that the random variable $Z:=\sup\{W_t-\delta t\colon t>0\}$ is exponentially  distributed with rate $2\delta$. It follows that $P(Z \ge x\mid Z\ge x_0)=e^{-2\delta(x-x_0)}$ and, by Lemma~\ref{lem:16}, the Average Value-at-Risk is 
\begin{equation*}
	\kappa(v)=-\E_{u_0} (Z\mid Z\le v) = -v+\int_0^\infty\frac{e^{-2\delta y}\mathrm{d}y}{e^{-2\delta v}} 
	=-\left(v-\frac{1}{2\delta}\right).
\end{equation*}
The result follows by substituting~\eqref{eq:delta} for $\delta$.
\end{proof}


\begin{remark}[Units]
Note that $X$ is measured in monetary units (\$, say). Then the unit of $%
\delta=\frac{1}{2\rho}\frac{\E X^2}{\E X}$ is in the same
monetary unit as well, which identifies~\eqref{eq:Capital} as a relevant economic quantity.
\end{remark}

\begin{remark}[Premium loadings, $\protect\rho$]
The premium loading $\rho$ is in the denominator of~\eqref{eq:Capital}. As a
consequence, the capital injection required in case after ruin is much
higher for small premium loadings. Although evident from economic
perspective, this is a plea for sufficient premium loadings and it is not
possible to recover a company in a competitive market where premium margins
are not sufficient.
\end{remark}

\begin{corollary}
The average capital required to recover a firm from ruin is 
\begin{equation}  \label{eq:kappa1}
\kappa=\kappa(0)=\frac{\E X^2}{2\rho\E X},
\end{equation}
even if the underlying risk process is driven by the fractional (FPP) or
non-homogeneous Poisson process, provided that $\Lambda(\cdot)$ is unbounded.
\end{corollary}

\begin{proof}
The process~\eqref{eq:17} with $\rho=0$ is symmetric and in the limit thus
is a time-changed Brownian motion. As in Corollary~\ref{cor:1} it is thus
sufficient to recall that the range of the process $Y_{\alpha}(t)$ is $%
\mathbb{R}_{>0}$ almost surely.
\end{proof}

\begin{remark}
	A further example of a risk functional satisfying the axioms~\ref{lab:i}--\ref{lab:iv} in Definition~\ref{def:Coherent} is the proportional Hazard risk measure (cf.\ Young \cite{Young}), which can be stated as 
	\begin{equation*}
		\rho_c(Z):=\int_0^\infty F_Z(y)^c dy
	\end{equation*}
	for a non-negative random variable $Y\ge 0$ and a parameter $c\in(0,1)$.

	The random variable $Z=\sup_{t>0} W_t-\delta t$ is exponentially distributed by~\eqref{eq:exp}. The proportional Hazard risk corresponding to the average capital required to recover a company thus is 
	\begin{equation*}
		\kappa_c:=\int_0^\infty F_Z^c(y)dy=\int_0^\infty e^{-2\delta c y}d y=
		\frac{1}{2\delta c}=\frac{1}{c}\frac{\E X^2}{2\rho\E X}=\frac\kappa c,
	\end{equation*}
	a simple multiple of~\eqref{eq:kappa1}.
\end{remark}

%
%
%
%

\subsection{The natural relation of the fractional Poisson process to the Entropic Value-at-Risk}

The premium of an insurance contract with random payoff $X\in\mathcal Z$ based on a risk functional $\rho$ is \begin{equation}\label{eq:rho1}\rho(X).\end{equation}
This assignment does not take dependencies into account. The fractional Poisson process has an incorporated, natural dependency structure (cf.\ Section~\ref{sec:4}, in particular~\ref{sec:42} above) and so it is natural to choose
\begin{equation}\label{eq:rho2}
\rho\left(\frac1n\sum_{i=1}^n X_i\right)
\end{equation} as a fair premium in a regime, where dependencies cannot be neglected; by convexity of the risk functional $\rho$ (cf.~\ref{lab:i} and~\ref{lab:iv} in Definition~\ref{def:Coherent}), the premium~\eqref{eq:rho1} is more expensive than~\eqref{eq:rho2}.
In a regime base on the fractional Poisson process, the claims up to time~$t$ are $N_\alpha(t)$ and thus the quantity \[S_{N_\alpha(t)}:=X_1+\dots+X_{N_\alpha(t)}\] is of particular interest.

\medskip
The Entropic Value-at-Risk has natural features, which combine usefully with the fractional Poisson process.

\begin{definition}For an $\mathbb R$\nobreakdash-valued random variable $Z\in\mathcal Z$, the entropic Value-at-Risk at risk level $\kappa\in[0,1)$ is
	\[\EVaR_{\kappa}(Z):=\sup \left\{\E YZ\colon Y\ge0, \E Y=1\text{ and } \E Y\log Z\le\frac{1}{1-\kappa}\right\}.\]
\end{definition}
The Entropic Value-at-Risk derives its name from the constraint $\E Y\log Y$, which is the entropy. The convex conjugate (i.e., the dual representation, cf.\ \cite{Pichler}) of the Entropic Value-at-Risk is given by
\begin{equation}\label{eq:EVaR}
	\EVaR_{\kappa}(Z):=\inf\left\{\frac{1}{h}\ln\frac1{1-\kappa}\E e^{h Z}\colon h>0\right\}.
\end{equation}

The representation~\eqref{eq:EVaR} bases the Entropic Value-at-Risk on the moment generating function.
This function is explicitly available for the fractional Poisson process. This allows evaluating
\[\EVaR\left(X_1+\dots+X_{N_\alpha(t)}\right) \] 
for specific $t>0$ based on the fractional Poisson process by a minimization exercise with a single variable via~\eqref{eq:EVaR}.

\begin{proposition}[Moment generating function]
	Assume that the random variables $X_1$, $X_2,\dots$ are iid. Then the moment
	generating function of the process loss $X_1+\dots+X_{N_\alpha(t)}$ with respect to the fractional Poisson process up to time $t$ is given explicitly by 
	\begin{equation}\label{eq:Moment}
	\E\left[e^{h(X_1+\dots+X_{N_{\alpha}(t)})}\right] =E_{\alpha}\Big(%
	\lambda t^{\alpha}\big(\varphi_X(h)-1\big)\Big),
	\end{equation}
	where $\varphi_X$ is the moment generating function of $X_1$, $\varphi_X(h):=%
	\E e^{h X}$.
\end{proposition}

\begin{proof}
	can be given by employing~\eqref{eq:26} as 
	\begin{align*}
	\E e^{h(X_1+\dots+X_{N_\alpha(t)})}&= \E \sum_{k=0}^\infty \E \left[%
	e^{h(X_1+\dots+X_k)}|\,N_t=k\right] \\
	&= \sum_{k=0}^\infty\frac{(\lambda t^{{\alpha} }\varphi_X(h))^{k}}{k!}%
	\mathop{\displaystyle \sum} \limits_{j=0}^\infty\frac{(k+j)!}{j!}\frac{%
		(-\lambda t^{{\alpha} })^{j}}{ \Gamma ({\alpha} (j+k)+1)}
	\end{align*}
	After formally substituting $j\leftarrow \ell-k$ we get 
	\begin{align*}
	\E e^{h(X_1+\dots+X_{N_\alpha(t)})}&=
	\sum_{\ell=0}^\infty\sum_{k=0}^\ell(\lambda t^{{\alpha} }\varphi_X(h))^{k}{%
		\binom{\ell}{k}}\frac{(-\lambda t^{{\alpha} })^{\ell-k}}{ \Gamma
		(\alpha\ell+1)} \\
	&= \sum_{\ell=0}^\infty\frac{(-\lambda t^{\alpha})^{\ell}}{ \Gamma
		(\alpha\ell+1)}\sum_{k=0}^\ell{\binom{\ell}{k}}(-\varphi_X(h))^{k} \\
	&= \sum_{\ell=0}^\infty\frac{(-\lambda t^{\alpha})^{\ell}}{ \Gamma
		(\alpha\ell+1)}(1-\varphi_X(h))^\ell \\
	&= E_{\alpha}\Big(-\lambda t^{\alpha}\big(1-\varphi_X(h)\big)\Big),
	\end{align*}
	which is the desired assertion.
\end{proof}

\subsection{The risk process at different levels of $\alpha$}

In what follows we compare the fractional Poisson process with the usual compounded Poisson process. To achieve a comparison on a \emph{like for like} basis the parameters of the processes are adjusted so that both have comparable claim numbers. The fractional Poisson process imposes a covariance structure which accounts for accumulation effects, which are not present for the classical compound Poisson process. Employing the Entropic Value-at-Risk as premium principle we demonstrate that the premium, which is associated with the fractional process, is always higher than the premium associated with the compound Poisson process.

To obtain the comparison on a like for like basis we recall from~\eqref{EXP} that the number of claims of the fractional Poisson process grows as $t^{\alpha}$, which is much faster than the claim number of a usual compound Poisson process and hence, the fractional Poisson process generates much more claims than the compound Poisson process at the beginning of the observation (cf.\ also Remark~\ref{rem:time}). To compare on a like for like basis we adjust the intensity or change the time so that the comparison is fair, that is, we assume that the claims of the classical compounded Poisson process triggered by $\lambda^\prime$ (denoted $N_1^{\lambda^\prime}$) and the fractional Poisson process $N_{\alpha}^{\lambda}$ are related by
\begin{equation}\label{eq:lambdas}
	\lambda^\prime t^\prime =\E N_1^{\lambda^\prime}(t^\prime)\le \E N_{\alpha}^{\lambda}(t)=\frac{\lambda t^{\alpha}}{\Gamma(1+\alpha)};
\end{equation}
note that equality in~\eqref{eq:lambdas} corresponds to the situation of Poisson processes at \emph{different}  $\alpha$ levels, but with an identical number of average claims up to time $t$ ($t^\prime$, resp.).

\begin{proposition}\label{prop:10}
	Let $\lambda^\prime t^\prime \le \frac{\lambda t^{\alpha}}{\Gamma(1+\alpha)}$ so that $\E N^{\lambda^\prime}_{1}(t) \le\E N^{\lambda}_{\alpha}(t)$ by~\eqref{EXP} and $X_i$ be iid.\ random variables with $X_i$ bounded. It holds that 
	\begin{equation}\label{eq:EVaRR}
		\EVaR_{\kappa}\left(S_{N_{1}^{\lambda^\prime}(t^\prime)}\right) \le \EVaR_{\kappa}\left(S_{N_{\alpha}^{\lambda}(t)}\right),\qquad\kappa\in[0,1),
	\end{equation}
	where $S_N=X_1+\dots+X_N$.
	
	If equality holds in~\eqref{eq:lambdas}, then the assertion~\eqref{eq:EVaRR} holds for all random variables (i.e., which are not necessarily nonnegative).
\end{proposition}

\begin{proof} 
	It holds that
	\begin{equation}\label{eq:33}
	\frac{\Gamma(1+\alpha k)}{\Gamma(1+\alpha)^{k}\cdot k!}\le 1;
	\end{equation} indeed, the statement is obvious for $k=0$ and $k=1$ and by induction we find that
	\begin{align}
	\Gamma\big(1+\alpha(k+1)\big)&=\alpha(k+1)\Gamma(\alpha+\alpha k)\nonumber\\
	&\le\alpha(k+1)\Gamma(1+\alpha k)\nonumber\\
	&\le \alpha(k+1)\Gamma(1+\alpha)^k k!\label{eq:hyp}\\
	&\le\Gamma(1+\alpha)^{k+1}(k+1)!,\nonumber
	\end{align}
	i.e,\ Equation~\eqref{eq:33}, where we have employed the induction hypothesis in~\eqref{eq:hyp}. We conclude that
	\[\frac{\big(\lambda t^{\alpha}\big)^k}{\Gamma(k\alpha+1)}
	\ge \frac{\big(\lambda t^{\alpha}\big)^k}{\Gamma(1+\alpha)^{k}k!}
	=\left(\frac{\lambda t^{\alpha}}{\Gamma(1+\alpha)}\right)^{k}\frac{1}{k!}
	\ge \frac{(\lambda^\prime t^\prime)^k}{k!}.
	\]
	From the moment generating function we have by~\eqref{eq:Moment} that \[\E e^{hS_{N_{\alpha}^{\lambda}}}= \sum_{k=0}^\infty \frac{\left(\lambda t^{\alpha}(\varphi_X(h)-1)\right)^k}{\Gamma(\alpha k+1)}.\] Now assume that $Y_i\ge0$ and thus $\E e^{hY_i}\ge \exp(\E hY_i)$ by Jensen's inequality and we have that $\varphi_Y(h)\ge1$ whenever $h>0$. The statement then follows from~\eqref{eq:EVaR} together with~\eqref{eq:Moment}.
\end{proof}
\begin{remark}
	We mention that the mapping $\alpha \mapsto E_{\alpha}\big(t z^{\alpha}\big)$ is not monotone in general, so the result in Propositon~\ref{prop:10} does not extend to a comparison for all~$\alpha$s.
\end{remark}

\section{Summary and conclusion}
\label{sec:7}
This paper introduces the fractional Poisson process and discusses its relevance in companies which are driven by some counting process. Empirical evidence supporting the fractional Poisson process is given by considering high frequency data and the classical Danish fire insurance data.
A natural feature of the fractional Poisson process is its dependence structure. We demonstrate that this dependence structure increases risk in a regime governed by the fractional Poisson process in comparison to the classical process.

The fractional Poisson process can serve as a stress test for insurance companies, as the number of claims is higher at its start. It can hence have importance for start ups and for re-insurers, which have a very particular, dependent claim structure in their portfolio.

It is demonstrated that quantities as the average capital to recover after ruin is independent of the risk level of the fractional Poisson process. Further, the Entropic Value-at-Risk is based on the moment generating function of the claims, and for this reason the natural combination of the fractional Poisson process and the Entropic Value-at-Risk is is demonstrated.

\section{Acknowledgment}
N. Leonenko was supported in particular by Australian Research Council's Discovery Projects funding scheme (project DP160101366) and by project MTM2015-71839-P of MINECO, Spain (co-funded with FEDER funds).


\begin{thebibliography}{99}
\bibitem{Pichler} Ahmadi-Javid, A. and Pichler, A. (2017) An analytical study of norms and Banach spaces induced by the entropic value-at-risk, Mathematics and Financial Economics vol.\ 11(4), 527--550

\bibitem{ALM} Aletti, G., Leonenko, N.\,N. and Merzbach, E. (2018) Fractional Poisson processes and martingales, Journal of Statistical Physics, 170, N4, 700--730.

\bibitem{adeh} Artzner, P., Delbaen, F., Eber, J-M and Heath, D. (1999) Coherent Measures of Risk, Mathematical Finance 9, 203--228


\bibitem{citekey} Beghin, L. and Macci, C. (2013) Large deviations for fractional Poisson processes, Stat. \& Prob. Letters, 83(4), 1193--1202,
doi:~\href{http://dx.doi.org/10.1016/j.spl.2013.01.017}{10.1016/j.spl.2013.01.017}

\bibitem{B01} Beghin, L. and Orsingher, E. (2009) Fractional Poisson processes and related planar random motions, Electron. J. Probab. 14, no. 61, 1790--1827.

\bibitem{BO2} Beghin, L. and Orsingher, E. (2010) Poisson-type processes governed by fractional and higher-order recursive differential equations. Electron. J.
Probab. 15, no. 22, 684--709.

\bibitem{BS} Biard, R., Saussereau, B. (2014) Fractional Poisson process: long-range dependence and applications in ruin theory. J. Appl. Probab. 51, no. 3, 727--740.

\bibitem{B} Bingham, N.\,H. (1971)  Limit theorems for occupation times of Markov processes, Z. Wahrscheinlichkeitstheorie und Verw. Gebiete 17, 1--22.

\bibitem{Borodin} Borodin, A.\,N. and Salminen, P. (2012) Handbook of Brownian motion-facts and formulae, Birkh{\"a}user, doi:10.1007/978-3-0348-7652-0.

\bibitem{CUW} Cahoy, D.\,O., Uchaikin, V.\,V. and Woyczynski, W.\,A. (2010) Parameter estimation for fractional Poisson processes, J. Statist. Plann. Inference,
140, no. 11, 3106--3120.

\bibitem{Cont2004} Cont, R. and Tankov, P. (2004) Financial Modeling With Jump Processes, Chapman \& Hall/ CRC.

\bibitem{DAL} Daley, D.\,J. (1999)  The Hurst index for a long-range dependent renewal processes, Annals of Probablity, 27, no 4, 2035-2041.

\bibitem{Grandell1991} Grandell, J. (1991) Aspects of Risk Theory, Springer, New York. 

\bibitem{HMS} Haubold, H.\,J., Mathai,  A.M. and Saxena, R.K. (2011) Mittag-Leffler functions and their applications. J. Appl. Math. Art. ID 298628, 51 pp.

\bibitem{KKV} Kataria, K. K. and Vellaisamy, P. (2018) On densities of the product, quotient and power of independent subordinators. J. Math. Anal. Appl. 462, 1627--1643.

\bibitem{KLS} Kerss, A. and Leonenko, N. N. and Sikorskii, A. (2014) Fractional Skellam processes with applications to finance. Fract. Calc. \& Appl. Analysis. 17, pp. 532--551.

\bibitem{Khinchin1969} Khinchin, A. Y. (1969) Mathematical Methods in the Theory of Queueing.
Hafner Publishing Co., New York.

\bibitem{KV} Kumar, A. and Vellaisamy, P. (2015) Inverse tempered stable subordinators. 
Statist. Probab. Lett. 103, pp. 134--141.

\bibitem{Kusuoka} Kusuoka, S. (2001) On law invariant coherent risk measures.  Advances in mathematical economics, Vol.\ 3 Ch.\ 4, Springer, 83--95

\bibitem{L} Laskin, N. (2003) Fractional Poisson process. Chaotic transport and complexity in classical and quantum dynamics, Commun. Nonlinear Sci. Numer.
Simul. 8, no. 3-4, 201--213.

\bibitem{LMS0} Leonenko, N.\,N., Meerschaert, M.\,M. and Sikorskii, A. (2013) Correlation structure of fractional Pearson diffusions, Computers and Mathematics and
Applications, 66, 737--745.

\bibitem{LMS00} Leonenko, N.\,N., Meerschaert, M.\,M. and Sikorskii, A. (2013) Fractional Pearson diffusions, Journal of Mathematical Analysis and Applications, 403, 532--246.

\bibitem{LMS} Leonenko, N.\,N., Meerschaert, M.\,M., Schilling, R.\,L. and
Sikorskii, A. (2014) Correlation structure of time-changed L\'{e}vy processes. Commun.
Appl. Ind. Math. 6, no. 1, e-483, 22 pp.

\bibitem{LM} Leonenko, N.\,N. and Merzbach, E. (2015) Fractional Poisson fields, Methodology and Computing in Applied Probability, 17, no.1, 155--168.

\bibitem{LST1} Leonenko, N.\,N., Scalas, E. and Trinh, M. (2019) Limit Theorems for the Fractional Non-homogeneous Poisson Process, J. Appl. Prob., In Press.

\bibitem{LST} Leonenko, N.\,N., Scalas, E. and Trinh, M. (2017) The fractional non-homogeneous Poisson process, Statistic and Probability Letters, 120, 147--156.

\bibitem{MGS} Mainardi, F., Gorenflo, R. and Scalas, E. (2004) A fractional generalization of the Poisson processes, Vietnam J. Math. 32, Special Issue, 53--64.

\bibitem{MGV1} Mainardi, F., Gorenflo, R. and Vivoli, A. (2005) Renewal processes of Mittag-Leffler and Wright type, Fract. Calc. Appl. Anal. 8, no. 1, 7--38.

\bibitem{MGV2} Mainardi, F., Gorenflo, R. and Vivoli, A. (2007) Beyond the Poisson renewal process: a tutorial survey. J. Comput. Appl. Math. 205, no.\ 2, 725--735.

\bibitem{MA} Malinovskii, V. (1998) Non-Poissonian claims' arrival and calculation of the ruin, Insurance Math. Econom. 22, 123--222.

\bibitem{MNV} Meerschaert, M.\,M., Nane, E. and Vellaisamy, P. (2011) The fractional Poisson process and the inverse stable subordinator, Electron. J. Probab. 16, no. 59, 1600--1620.

\bibitem{MS1} Meerschaert, M.\,M. and Sikorskii, A. (2012) Stochastic Models for Fractional Calculus, De Gruyter, Berlin/Boston.

\bibitem{Mikosch2009} Mikosch, T. (2009) Non-life insurance mathematics, an introduction with the Poisson process. Springer.

\bibitem{PR} Pflug, G.\,Ch. and R\"omisch, W. (2007) Modeling, Measuring and Managing Risk, World Scientific, doi: 10.1142/9789812708724

\bibitem{P} Podlubny, I. (1999) Fractional differential equations. An introduction to fractional derivatives, fractional differential equations, to methods of their solution and some of their applications, Mathematics in Science and Engineering, 198, Academic Press, Inc., San Diego, CA.

\bibitem{Raberto2002} Raberto, M., Scalas, E. and Mainardi, F. (2002) Waiting times and returns in high-frequency financial data: an empirical study, Physica A, 314, 749--755.

\bibitem{RS} Repin, O.\,N. and Saichev, A.\,I. (2000)  Fractional Poisson law, 
Radiophysics and Quantum Electronics, 43, no. 9, 738--741.

\bibitem{SamorodnitskyTaqqu} Samorodnitsky, G. and Taqqu, M.\,S. (1994) Stable Non-Gaussian Random Processes. Chapman \& Hall, New York.

\bibitem{Scalas2004} Scalas, E., Gorenflo, R., Luckock, H.,  Mainardi, F., Mantelli, M. and Raberto, M. (2004). Anomalous waiting times in high-frequency financial data, Quant. Financ. 4, 695--702.

\bibitem{VTM1} Veillette, M. and Taqqu, M.\,S. (2010) Numerical computation of first passage times of increasing L\'{e}vy processes, Methodol. Comput. Appl.
Probab. 12, no. 4, 695--729.

\bibitem{VTM2} Veillette, M. and Taqqu, M.\,S. (2010) Using differential equations to
obtain joint moments of first-passage times of increasing L\'{e}vy
processes, Statist. Probab. Lett. 80, no. 7-8, 697--705.

\bibitem{WW} Wang, X.-T. and Wen, Z.-X. (2003) Poisson fractional processes, Chaos
Solitons Fractals, 18, no. 1, 169--177.

\bibitem{WWZ} Wang, X.-T., Wen, Z.-X. and Zhang, S.-Y. (2006) Fractional Poisson
process. II, Chaos Solitons Fractals, 28, no. 1, 143--147.

\bibitem{WZF} Wang, X.-T., Zhang,  S.-Y. and Fan, S. (2007) Nonhomogeneous fractional
Poisson processes. Chaos Solitons Fractals, 31, no. 1, 236--241.

\bibitem{Young} Young, V.\,R. (2006) Premium Principles in \emph{Encyclopedia of Actuarial Science}.
\end{thebibliography}
\end{document}